\documentclass[12pt]{amsart}
\usepackage{amssymb,amsthm,graphicx,caption,mwe,float,quiver,adjustbox, comment}%

\makeatletter
\newcommand*\rel@kern[1]{\kern#1\dimexpr\macc@kerna}
\newcommand*\widebar[1]{%
  \begingroup
  \def\mathaccent##1##2{%
    \rel@kern{0.8}%
    \overline{\rel@kern{-0.8}\macc@nucleus\rel@kern{0.2}}%
    \rel@kern{-0.2}%
  }%
  \macc@depth\@ne
  \let\math@bgroup\@empty \let\math@egroup\macc@set@skewchar
  \mathsurround\z@ \frozen@everymath{\mathgroup\macc@group\relax}%
  \macc@set@skewchar\relax
  \let\mathaccentV\macc@nested@a
  \macc@nested@a\relax111{#1}%
  \endgroup
}
\makeatother

\newtheorem{thm}{Theorem}
\newtheorem{theorem}{Theorem}[section]
\newtheorem{corollary}[theorem]{Corollary}
\newtheorem{proposition}[theorem]{Proposition}
\newtheorem{lemma}[theorem]{Lemma}

\title [The commuting graph]{The commuting graph and a graph associated with centralizers}
\author[Lewis]{Mark L. Lewis}
\address{Department of Mathematical Sciences, Kent State University, Kent, OH  44242; lewis@math.kent.edu}
\author[McCulloch]{Ryan McCulloch}
\address{Department of Mathematics and Statistics, Binghamton University, Binghamton, NY 13902; rmccullo1985@gmail.com}
\date{}

\begin{document}
	
\begin{abstract}
Let $G$ be a $p$-group.  We begin to consider the relationship between the structure of the commuting graph and $|G:Z(G)|$.   We also build a family of groups whose commuting graphs have more than one connected component whose diameter is at least $2$.  For this, we introduce another graph related to the commuting graph that is associated with centralizers.%		Not needed.
\end{abstract}


\maketitle

	%\begin{document}

\section{Introduction}

Throughout this note all groups will be finite.  Let $G$ be a group.  A number of graphs have been attached to groups and lately the number of different graphs considered have exploded.  We recommend the reader consult the excellent surveys \cite{camsurv1}, \cite{camInd}, and \cite{camsurv2} for some of the background on many of the graphs that have been studied recently.   The {\it commuting graph} ${\mathfrak C}(G)$ for $G$ is the (simple) graph whose vertices are $G \setminus Z(G)$ and if $a,b \in G \setminus Z(G)$, $a\neq b$, then there is an edge between $a$ and $b$ if $ab = ba$.  The complement of $G$ is called the noncommuting graph.  As far as we can determine, the commuting graph was first introduced by Bertram in \cite{bertram}.  The commuting and noncommuting graphs have been studied in a number of contexts.  %In his landmark work on classifying $p$-groups, P. Hall made the definition of isoclinism and showed it was a generalization of isomorphism.  We will state the definition of isoclinism in Section \ref{sec:some}.  In our first result, we will show that isoclinic groups with the same order have isomorphic commuting graphs.  Perhaps the most important examples of isoclinic groups are that the dihedral groups and dicyclic groups of the same order are isoclinic.  Hence, this first theorem shows that the dihedral groups and dicyclic groups of the same orders have isomorphic commuting graphs.  

%\begin{thm} \label{intro thm 1}
%If $G_1$ and $G_2$ are isoclinic groups with $|G_1| = |G_2|$, then $\mathfrak{C} (G_1) \cong \mathfrak{C} (G_2)$.	
%\end{thm}

In this note, we wish to consider the diameter of these graphs.  In \cite{symmetric}, Iranmanesh and Jafarzadeh conjecture for all groups that either the commuting graph is disconnected or there is a universal constant that bounds the diameter.  However, Guidici and Parker proved in \cite{nobound} that this conjecture is false by presenting a family of $2$-groups having nilpotence class $2$ with the property that for every integer $n$ there is a group in the family whose commuting graph has diameter at least $n$.  In her dissertation, Rachel Carleton has shown that for every prime $p$ there is such a family of groups \cite{carldiss}. 

In this paper, we want to show when $G$ is a $p$-group that there is a relationship between the diameter of $\mathfrak{C} (G)$ and $|G:Z (G)|$.  The first theorem is true for all groups, not just $p$-groups.

\begin{thm} \label{intro thm 2}
Let $G$ be a group.  If $|G'| < |G:Z(G)|^{1/2}$, then ${\mathfrak C} (G)$ is connected and has diameter at most $2$.
\end{thm}

We next want to begin to show when $G$ is a $p$-group that the structure of $\mathfrak{C} (G)$ is determined by $|G:Z (G)|$.  

\begin{thm} \label{intro thm 3}
Let $p$ be a prime and let $G$ be a $p$-group.  Then the following are true:
\begin{enumerate}
\item If $|G:Z(G)| \le p^3$, then every connected component of $\mathfrak{C} (G)$ is a complete graph.
\item If $p^4 \le |G:Z(G)| \le p^5$, then either every connected component of $\mathfrak{C} (G)$ is a complete graph or there is one connected component whose diameter is at least $2$ and its diameter is at most $4$.	 
\end{enumerate}
\end{thm}

When $G$ is a $p$-group of nilpotence class $n$ where $n \ge 3$, there is another index which can give some evidence that influences the structure of $\mathfrak{C} (G)$.  Define $G_{n-1}$ to be the $(n-1)$\textsuperscript{st} term in the lower central series for $G$.  

\begin{thm} \label{intro thm 4}
Let $p$ be a prime, let $n \ge 3$ be an integer, and let $G$ be a $p$-group of nilpotence class $n$.  If $|G:C_G (G_{n-1})|$ is $p$ or $p^2$, then $\mathfrak{C} (G)$ has at most one connected component that is not complete and it has diameter at most $8$.
\end{thm}

Finally, all of the groups in the literature have the property that if the commuting graph is disconnected, then either all of the connected components are complete or all but one of the connected components are complete and there is one connected component that is not complete.  We will build a family of examples of groups whose commuting graphs have more than one connected component that is not a complete graph.  In fact, we will be able to construct groups whose commuting graphs have a specified number of noncomplete connected components and we can specify diameters of each of the connected components.

We close by mentioning that we accomplish the work in this paper by introducing a new graph on the centralizers of the group and then proving that our graph on the centralizers and the commuting graph has correspondence on connected components that preserves diameter except for complete connected components which correspond to isolated vertices.

\section {Centers and Centralizers}

If $g \in G \setminus Z(G)$, we will write $Z (g) = Z (C_G (g))$.  We write $\mathcal {Z} (G) = \{ Z(g) \mid g \in G \setminus Z(G) \}$ and $\mathcal{C} (G) = \{ C_G (g) \mid g \in G \setminus Z(G) \}$.  

The following lemma has been proved as Lemma 2.3 in \cite{max abel}.

\begin{lemma} \label{three}
Let $G$ be a group and suppose $a,b \in G \setminus Z(G)$.
\begin{enumerate}
\item $a \in C_G(b)$ if and only if $Z(a) \le C_G (b)$.
\item $Z(a) \le C_G (b)$ if and only if $Z(b) \le C_G (a)$.
\end{enumerate}
\end{lemma}

\begin{lemma}\label{three one}
Let $G$ be a group and suppose $a,b \in G \setminus Z(G)$.  Then $a \in Z(b)$ if and only if $Z(a) \le Z(b)$.    
\end{lemma}
    
\begin{proof}
Suppose $a \in Z (b)$.  This implies that $a \in Z(C_G (b))$.  Thus, $a$ centralizes $C_G (b)$.  Hence, $C_G (b) \le C_G (a)$.  It follows that $Z(a) = Z(C_G (a)) \le Z(C_G (b)) = Z(b)$.  Conversely, suppose that $Z(a) \le Z(b)$.  We have $a \in Z(a)$, so $a \in Z(b)$.
\end{proof}

The following lemma is partly an immediate observation and partly folklore that we gave a proof of in Lemma 3.2 in \cite{max abel}.

\begin{lemma} \label{abel cent}
Let $G$ be a group and suppose $a \in G \setminus Z(G)$.  Then the following are equivalent:
\begin{enumerate}
\item $C_G (a) = Z (a)$.
\item $C_G (a)$ is abelian.
\item $C_G (a)$ is maximal among abelian subgroups of $G$.
\end{enumerate}
\end{lemma}

\begin{proof}
If $C_G(a) = Z(a)$, then $C_G (a) = Z(C_G (a))$, so obviously $C_G (a)$ is abelian.  We showed in Lemma 3.2 of \cite{max abel} that $C_G (a)$ abelian implies that $C_G (a)$ is maximal abelian.  Finally, if $C_G (a)$ is maximal abelian, then it is abelian, and $C_G (a) = Z(C_G (a)) =Z(a)$.    
\end{proof}

It is not difficult to see that every element in $G \setminus Z(G)$ lies in some $Z(g)$.  We now show that the centralizers of $G$ are covered by these sets.

\begin{corollary}
Let $G$ be a group, and let $a \in G \setminus Z(G)$.  Then $$C_G (a) = \bigcup_{b \in C_G(a) \setminus Z(G)} Z(b).$$
\end{corollary}

\begin{proof}
By Lemma \ref{three} (1), we have if $b \in C_G (a) \setminus Z(G)$, then $Z(b) \le C_G(a)$, so $\displaystyle \bigcup_{b \in C_G(a) \setminus Z(G)} Z(b) \le C_G (a)$.  But also, $b \in Z(b) \setminus Z(G)$, so $\displaystyle  C_G (a) \setminus Z(G) \subseteq \bigcup_{b \in C_G(a) \setminus Z(G)} Z(b) \setminus Z(G)$.  It follows that $\displaystyle  C_G (a) \subseteq \bigcup_{b \in C_G(a) \setminus Z(G)} Z(b)$, and the desired equality follows.
\end{proof}

We now have a double centralizer condition.

\begin{lemma} \label{cent equality}
Let $G$ be a group and let $g \in G \setminus Z(G)$.  Then $C_G (g) = C_G (Z(g))$.
\end{lemma}

\begin{proof}
Since $Z(g) = Z(C_G (g))$, we have $C_G (g) \le C_G (Z(g))$.  On the other hand, since $g \in Z(g)$, we see that the elements of $C_G (Z(g))$ centralize $g$, and so $C_G (Z(g)) \le C_G (g)$.  Thus, the equality holds.     
\end{proof}

The next lemma is key. 

\begin{lemma} \label{star 1}
Let $G$ be a group and let $g, h \in G \setminus Z(G)$.  Then 
\begin{enumerate}
\item $C_G (g) \le G_G (h)$ if and only if $Z(h) \le Z(g)$.
\item $C_G (g) = C_G (h)$ if and only if $Z (g) = Z(h)$.
\end{enumerate}
\end{lemma}

\begin{proof}
Suppose that $C_G (g) \le C_G (h)$.   Then $Z (h) = Z (C_G (h)) \le Z (C_G (g)) = Z(g)$.  Conversely, suppose that $Z(h) \le Z (g)$.  By Lemma \ref{cent equality}, we know that $C_G (g) = C_G (Z(g)) \le C_G (Z(h)) = Z(h)$.  Notice that (2) follows from several applications of (1). 
\end{proof}

We write $\mathcal {Z} (G) = \{ Z(g) \mid g \in G \setminus Z(G) \}$ and $\mathcal{C} (G) = \{ C_G (g) \mid g \in G \setminus Z(G) \}$.  Notice that Lemma \ref{star 1} (2) gives us a bijection between $\mathcal{Z} (G)$ and $\mathcal {C} (G)$, and Lemma \ref{star 1} (1) implies that this bijection is containment reversing.

We now write formally in this next corollary.

\begin{corollary}\label{lem:corres}
Let $G$ be a group. %, let 
	%$$\mathcal {C} (G) = \{ {C}_G(a) \mid a \in G \setminus Z(G) \} ~~{\rm and} ~~\mathcal {Z} (G) = \{ {Z}(a) \mid a \in G \setminus Z (G) \}.$$  
Then the map $\phi : \mathcal {C} (G) \rightarrow \mathcal {Z} (G)$ given by 
$$\phi(X) = Z (X) = C_G (X)$$
is an order-reversing $1$-to-$1$ correspondence with $\phi^{-1} : \mathcal {Z} (G) \rightarrow \mathcal {C} (G)$ given by $$\phi^{-1} (Y) = {C}_G (Y).$$
\end{corollary}

We note that $\mathcal {Z} (G)$ and $\mathcal {C} (G)$ are not usually lattices, but we obtain the following for intersections.

\begin{lemma} \label{intersection}
Let $G$ be a group and let $g, h \in G \setminus Z(G)$.   Then either $Z(g) \cap Z(h) = Z(G)$ or $\displaystyle Z(g) \cap Z(h) = \prod Z(c)$ where $c$ runs over the elements in $(Z (g) \cap Z(h)) \setminus Z(G)$.
\end{lemma}

\begin{proof}
If $Z(g) \cap Z(h) = Z(G)$, then we are done.  Thus, we assume $Z(g) \cap Z(h) > Z(G)$.  Suppose $c \in (Z (g) \cap Z(h)) \setminus Z(G)$.  By Lemma \ref{three one}, we have $Z(c) \le Z(g)$ and $Z(c) \le Z(h)$, so $Z(c) \le Z(g) \cap Z(h)$.  Since $c$ was arbitrary, we have $\displaystyle \prod Z(c) \le Z(g) \cap Z(h)$ where $c$ runs over all the elements in $(Z(g) \cap Z(h)) \setminus Z(G)$.  On the other hand, observe that since $c \in Z(c)$ for every $c \in (Z(g) \cap Z(h)) \setminus Z(G)$, we have $\displaystyle \displaystyle (Z(g) \cap Z(h)) \setminus Z(G) \subseteq \prod Z(c)$.  Since $Z(G) \le Z(c)$ for every $c$, it follows that $\displaystyle Z(g) \cap Z(h) \le \prod Z(c)$.  Thus, equality holds.    
\end{proof}

Expanding on the work of It\^o in \cite{ito}, Rebmann in \cite{reb} defines a group $G$ to be an {\it $F$-group} if for all $x,y \in G \setminus Z(G)$ the condition $C_G(x) \le C_G (y)$ implies $C_G (x) = C_G (y)$.  Following \cite{DHJ}, we say that a group $G$ is a CA-group if $C_G (g)$ is abelian for all $g \in G \setminus Z(G)$.  It is not difficult to see that every CA-group will be an $F$-group.  However, there are many $F$-groups that are not CA-groups.  Recall that Suzuki has classified the groups where the centralizers of all nonidentity elements are abelian.  In particular, it is shown in \cite{bsw} and \cite{suz} that if $G$ is a nonabelian group where all the centralizers are abelian, then either $G$ is ${\rm PSL} (2,2^m)$ for some integer $m \ge 2$ or $G$ is a Frobenius group whose Frobenius kernel and Frobenius complement are abelian.

\section{Commuting graphs}\label{sec:some}

In this section, we introduce a graph that is closely related to the commuting graph.  

Let $G$ be a group.  The commuting graph $\mathfrak {C} (G)$ for $G$ is the graph whose vertices are $G \setminus Z(G)$ and if $a,b \in G \setminus Z(G)$, $a \neq b$, then there is an edge between $a$ and $b$ if $ab = ba$.  

We first consider when $\mathfrak{C} (G_1) \cong \mathfrak{C} (G_2)$.  It is not difficult to see that $|G_1| = |G_2|$ is equivalent to $|G_1:Z(G_1)| = |G_2:Z(G_2)|$ which is equivalent to $|Z(G_1)| = |Z(G_2)|$.   However, it is not obvious that these equalities have to hold in general.   We will show that if we make the additional assumption that $G_1$ and $G_2$ are both $p$-groups for some prime $p$, then these equalities hold.

\begin{lemma}\label{lem:C_sizes}
Let $G_1$ and $G_2$ be groups so that $\mathfrak{C} (G_1) \cong \mathfrak{C} (G_2)$.  Then the following are equivalent:
\begin{enumerate}
\item $|G_1| = |G_2|.$
\item $|G_1:Z(G_1)| = |G_2:Z(G_2)|$.
\item $|Z(G_1)| = |Z(G_2)|$.
\end{enumerate}
Furthermore, if there is a prime $p$ so that $G_1$ and $G_2$ are also both $p$-groups, then these three equations necessarily hold.
\end{lemma}

\begin{proof}
Observe that since $\mathfrak C(G_1) \cong \mathfrak C(G_2)$, we have $|G_1| - |Z(G_1)| = |G_2| - |Z(G_2)|$.  This implies that $$|Z(G_1)| (|G_1:Z(G_1)| - 1) = |Z(G_2)| (|G_2:Z(G_2)| - 1).$$  The equivalence of the three equations is now immediate.  

Assume now that there is a prime $p$ so that $G_1$ and $G_2$ are $p$-groups.  For each $i$, observe that $|Z(G_i)|$ will be the full $p$-power dividing $|Z(G_i)| (|G_i:Z(G_i)| - 1)$.  Hence, $|Z(G_1)|$ will equal $|Z(G_2)|$, and this proves the lemma.
\end{proof}

%The first graph will be defined on cosets of the center.  %The second graph will defined on centralizers.  We will show that both of these graphs are invariant under isoclinisms.
The following variation on the commuting graph appears in \cite{d2n}.  We define the graph $\mathfrak {C}^* (G)$ to be the graph obtained by taking the subgraph of $\mathfrak {C} (G)$ induced by a transversal for $Z (G)$ in $G$.  It is not difficult to see that the graph obtained is independent of the transversal chosen since $a$ and $b$ commute if and only if $a z_1$ and $b z_2$ commute for all $z_1, z_2 \in Z(G)$. 

We will show that for groups of the same order, the commuting graphs are isomorphic if and only if the graphs on the transversals are isomorphic.
%Suppose that $G_1$ and $G_2$ are two groups. 

\begin{theorem}\label{thm:C=C}
Let $G_1$ and $G_2$ be groups so that $|G_1| = |G_2|$.  Then $\mathfrak C(G_1) \cong \mathfrak C(G_2)$ if and only if $\mathfrak {C}^* (G_1) \cong \mathfrak {C}^* (G_2)$.
\end{theorem}

Let $\Gamma$ be a graph. If $u$ is a vertex of $\Gamma$, then we use $N(u)$ to denote the neighbors of $u$.  I.e., $N(u)$ is the set of vertices in $\Gamma$ that are adjacent to $u$.  We define an equivalence relation on the vertices of $\Gamma$.  We say that $u {\sim} v$ if either $u = v$ or $u$ is adjacent to $v$ and $\{u\} \cup N(u) = \{v\} \cup N(v)$.  (We mention that adjacent vertices $u$ and $v$ having the property that $\{u\} \cup N(u) = \{v\} \cup N(v)$ are known as {\it clones} or {\it true twins} or {\it closed twins} in the literature.)  We can then form the quotient graph $\Gamma/{\sim}$.  The vertices of this graph are the equivalence classes under ${\sim}$.  If $[u]$ and $[v]$ are the equivalence classes of $u$ and $v$, then $[u]$ and $[v]$ are adjacent in $\Gamma/{\sim}$ if and only if $u$ and $v$ are adjacent in $\Gamma$.  Observe that ${\sim}$ is uniquely determined by $\Gamma$.  Hence, if $\Gamma$ and $\Delta$ are isomorphic graphs, then $\Gamma/{\sim}$ and $\Delta/{\sim}$ will be isomorphic. 

Before we prove Theorem \ref{thm:C=C}, we need a lemma that allows us to deal with twin vertices.  Intuitively, this lemma establishes that an isomorphism between graphs is ``free'' on twin classes, in the sense that it can take on arbitrary values within a twin class.

\begin{lemma}\label{lem:graph_iso}
Given a graph $\Gamma$, let ${\sim}$ be the equivalence relation induced by closed twins.  For each equivalence class $E \in {\sim}$, let $S_{E}$ be the symmetric group on the set $E$ and let $\sigma_E \in S_{E}$ be arbitrary.  If $f$ is a graph isomorphism from $\Gamma_1$ to $\Gamma_2$, then the map $f^*$ defined by $f^*(x) = f(\sigma_{E}(x))$ for $x$ in class $E$ of $\Gamma_1$ is a graph isomorphism from $\Gamma_1$ to $\Gamma_2$. 
\end{lemma}

\begin{proof}
Note that $x$ and $y$ are adjacent in $\Gamma_1$ if and only if $\sigma_E(x)$ and $\sigma_F(y)$ are adjacent in $\Gamma_1$, where $E$ is the twin class of $x$ and $F$ is the twin class of $y$.  This is true since $x$ and $y$ are adjacent if and only if a member of $E$ is adjacent to a member of $F$.

Hence $x$ and $y$ are adjacent in $\Gamma_1$ if and only if $\sigma_E(x)$ and $\sigma_F(y)$ are adjacent in $\Gamma_1$ if and only if $f(\sigma_E(x))$ and $f(\sigma_F(y))$ are adjacent in $\Gamma_2$ if and only if $f^*(x)$ and $f^*(y)$ are adjacent in $\Gamma_2$.
\end{proof}

We are now ready to prove Theorem \ref{thm:C=C}.

\begin{proof}[Proof of Theorem \ref{thm:C=C}]
Let $f$ be an isomorphism from $\mathfrak{C}(G_1)$ to $\mathfrak{C}(G_2)$.  Note that $E$ is a twin class in $\mathfrak{C}(G_1)$ if and only if $f(E)$ is a twin class in $\mathfrak{C}(G_2)$.  

By Lemma \ref{lem:C_sizes}, we see that $\mathfrak{C} (G_1) \cong \mathfrak{C} (G_2)$ and $|G_1| = |G_2|$ imply that $|Z(G_1)| = |Z(G_2)|=m$.  Now twin classes $E$ and $f(E)$ have the same size, and this size is a multiple of $m$ since $x$ and $xz$ are twins for any element $z$ in the center.  It follows that the number of elements in $E$ that represent distinct transversal classes is equal to the number of such elements in $f(E)$.  Let $t$ be a bijection from the vertices of $\mathfrak{C}^*(G_2)$ in $f(E)$ to the vertices of $\mathfrak{C}^*(G_1)$ in $E$.  One can then take the $1$-to-$1$ map $t(b) \mapsto f^{-1}(b)$ and extend it to a permutation $\sigma_E$ on all of $E$.

Build such $\sigma_E$'s for all twin classes $E$, and by Lemma \ref{lem:graph_iso}, we form the map $f^*$ which is an isomorphism from $\mathfrak{C}(G_1)$ to $\mathfrak{C}(G_2)$.

Restrict $f^*$ to obtain a map from $\mathfrak {C}^* (G_1)$ to $\mathfrak {C}^* (G_2)$.  Observe that $(a_i,a_j)$ is an edge in $\mathfrak {C}^* (G_1)$ if and only if $(a_i,a_j)$ is an edge in $\mathfrak C(G_1)$.  Now, $(a_i,a_j)$ is an edge in $\mathfrak C(G_1)$ if and only if $(f^* (a_i),f^*(a_j))$ is an edge in $\mathfrak C(G_2)$.  Finally, we see that $(f^* (a_i), f^* (a_j))$ is an edge in $\mathfrak C(G_2)$ if and only if $(f^* (a_i), f^* (a_j))$ is an edge in $\mathfrak {C}^* (G_2)$.  Hence if $\mathfrak C(G_1) \cong \mathfrak C(G_2)$ and $G_1$ and $G_2$ satisfy any of the equivalent conditions of Lemma \ref{lem:C_sizes}, then $\mathfrak {C}^* (G_1) \cong \mathfrak {C}^* (G_2)$. 

Conversely, let $f$ be an isomorphism from $\mathfrak{C}^*(G_1)$ to $\mathfrak{C}^*(G_2)$.  We are assuming $|G_1| = |G_2|$ and observe that $\mathfrak{C}^* (G_1) \cong \mathfrak{C}^* (G_2)$ implies that $|G_1:Z(G_1)| = |G_2:Z(G_2)|$, and it follows that $|Z(G_1)| = |Z(G_2)|$. Let $t$ be a bijection from $Z(G_1)$ to $Z(G_2)$. Define a map $g$ from $\mathfrak{C}(G_1)$ to $\mathfrak{C}(G_2)$ by $g(xz) = f(x)t(z)$ for every vertex $x$ in $\mathfrak{C}^*(G_1)$ and for every $z \in Z(G_1)$.  Let $x,y$ be vertices of $\mathfrak{C}^*(G_1)$ and $z_1,z_2 \in Z(G_1)$.  Observe that $(xz_1,yz_2)$ is an edge in $\mathfrak {C} (G_1)$ if and only if either $x=y$ and $z_1 \neq z_2$ or $(x,y)$ is an edge in $\mathfrak {C}^*(G_1)$ if and only if either $f(x)=f(y)$ and $z_1 \neq z_2$ or $(f(x),f(y))$ is an edge in $\mathfrak {C}^*(G_2)$ if and only if $(f(x)t(z_1),f(y)t(z_2))$ is an edge in $\mathfrak {C}(G_2)$.  Hence if $\mathfrak {C}^* (G_1) \cong \mathfrak {C}^* (G_2)$ and $|G_1| = |G_2|$, then $\mathfrak C (G_1) \cong \mathfrak C (G_2)$. 
\end{proof}

%It is immediate to see that if $\mathfrak C(G_1) \cong \mathfrak C(G_2)$, then $\mathfrak {C}^* (G_1) \cong \mathfrak {C}^* (G_2)$.  Conversely, if $\mathfrak {C}^* (G_1) \cong \mathfrak {C}^* (G_2)$ and $|G_1| = |G_2|$, then we also have $\mathfrak C (G_1) \cong \mathfrak C (G_2)$.    

%\section{Isoclinic groups}

We now develop the relationship between isoclinic groups and $\mathcal{C}$.  Following the literature, we say that $G_1$ and $G_2$ are {\it isoclinic} if there exist isomorphisms $$\alpha : G_1/Z(G_1) \rightarrow G_2/Z(G_2) {\rm ~~~~ and~~~~} \beta : G_1' \rightarrow G_2'$$ so that $$\beta ([g,h]) = [\alpha(g Z(G_1)), \alpha(h Z(G_1))]$$ for all $g,h \in G_1$.  If one is not used to working with isoclinisms, this last equation deserves an explanation.  The point is that the commutator $[\alpha (gZ (G_1)),\alpha (hZ(G_1))]$ has the same value no matter what element in the coset $Z(G_2)$ is used.  Hence, we obtain a well-defined element in $(G_2)'$, and this equation really does make sense.  

Suppose $G_1$ and $G_2$ are isoclinic.  Then it is easy to see that $\alpha$ will map a transversal for $Z(G_1)$ in $G_1$ to a transversal for $Z(G_2)$ in $G_2$, and the commutator condition will imply that a pair of cosets in $G_1$ commutes if and only if the corresponding pair of cosets in $G_2$ commute.  As we were making final preparations of this paper, we came across \cite{camsurv2} and saw that in Proposition 5.1 of that paper, they prove a result similar to our Theorem \ref{intro thm 1}.  However, our argument is more detailed and since we wanted to use our proof and the related arguments in the rest of the paper, we decided to retain our proof.

We will now prove the following:

\begin{theorem} \label{isoclinic}
If $G_1$ and $G_2$ are isoclinic groups, then $\mathfrak {C}^* (G_1) \cong \mathfrak {C}^* (G_2)$.
\end{theorem}

\begin{proof}
Let $\alpha : G_1/Z(G_1) \rightarrow G_2/Z (G_2)$ be the isomorphism from our isoclinism.  Let $a_1, \dots, a_n$ be the transversal for $Z(G_1)$ in $G_1$ given by $\mathfrak{C}^*(G_1)$, where the representative for the center is removed, and so $n = |G_1 : Z(G_1)| - 1$.  Order the transversal for $Z(G_2)$ in $G_2$ given by $\mathfrak{C}^*(G_2)$ so that $b_i \in \alpha (a_i Z(G_1))$.  Define $\phi : \mathfrak {C}^* (G_1) \rightarrow \mathfrak {C}^* (G_2)$ by $\phi (a_i) = b_i$.  We see that there is an edge in $\mathfrak {C}^* (G_1)$ for $(a_i,a_j)$ if and only if $a_i$ and $a_j$ commute.  Since $G_1$ and $G_2$ are isoclinic, we see that $a_i$ and $a_j$ commute if and only if $b_i$ and $b_j$ commute.  Also, $b_i$ and $b_j$ commute if and only if $(b_i,b_j)$ is an edge in $\mathfrak {C}^* (G_2)$.  It follows that $(\phi (a_i),\phi (a_j))$ will be an edge if and only if $(a_i,a_j)$ is an edge.  Therefore, we conclude that $\phi$ is a graph isomorphism.	
\end{proof}

%Hence, if $G_1$ and $G_2$ are isoclinic, then $\mathfrak {C}^* (G_1) \cong \mathfrak {C}^* (G_2)$.

%To close this section, we now have everything we need to prove Theorem \ref{intro thm 1}.
To close this section, we prove that isoclinic groups of the same order have isomorphic commuting graphs.   As we said above, this appears as Proposition 5.1 in \cite{camsurv2}.

\begin{theorem} \label{intro thm 1}
If $G_1$ and $G_2$ are isoclinic groups with $|G_1| = |G_2|$, then $\mathfrak{C} (G_1) \cong \mathfrak{C} (G_2)$.	
\end{theorem}

\begin{proof}%[Proof of Theorem \ref{intro thm 1}]
Suppose that $G_1$ and $G_2$ are isoclinic with $|G_1| = |G_2|$.  By Theorem \ref{isoclinic}, we see that $\mathfrak{C}^* (G_1) \cong \mathfrak{C}^* (G_2)$.   Since $|G_1| = |G_2|$, we may use Theorem \ref{thm:C=C} to see that $\mathfrak{C} (G_1) \cong \mathfrak{C} (G_2)$.   
\end{proof}

It is noted in \cite{camsurv2} that the converse of Theorem \ref{intro thm 1} is false.  A counterexample is given in Section 6 of that paper.  However, it is conjectured in \cite{ArCa} that it is true for groups of nilpotence class $2$.  We want to give some non-nilpotent examples.  We also obtain counterexamples by considering Frobenius groups of order $p^2 q$ where $p$ and $q$ are primes so that $q$ divides $p + 1$.  In both cases, the Frobenius kernel has order $p^2$ and a Frobenius complement is cyclic of order $q$.  In the first case, take the Frobenius kernel to be cyclic and in the second case, take the Frobenius kernel to be elementary abelian.  It is not difficult to see that these groups are not isolinic but that their commuting graphs are isomorphic.

\section{Centralizer Graphs}

Next, we introduce another graph that we will show is closely related to the commuting graph.

Let $G$ be a group.  We now define the {\it centralizer graph}, $\Gamma_{\mathcal Z} (G)$, to be the graph with vertices $\{ Z \in {\mathcal Z} (G) \}$.  If $Z_1, Z_2 \in {\mathcal Z} (G)$ with $Z_1 \ne Z_2$, then there is an edge between $Z_1$ and $Z_2$ precisely when $Z_2 \le C_G (Z_1)$.  Notice via Lemma \ref{three} that $Z_2 \le C_G (Z_1)$ if and only if $Z_1 \le C_G (Z_2)$.  Hence, it makes sense to think of $\Gamma_{\mathcal Z} (G)$ as an undirected graph.  Recall that ${\mathcal C} (G)$ is in bijection with ${\mathcal Z} (G)$, so we could have used $\{ C \in {\mathcal C} (G) \}$ for our vertex set.  There may be some times where this identification with the vertices will be preferable. %We set $\Gamma_{\mathcal Z}^* (G)$ to be the induced subgraph for the vertices $\{ Z \in {\mathcal Z} (G) \mid Z < C_G (Z) \}$.

We next show that $\Gamma_{\mathcal Z} (G)$ can be obtained from either $\mathfrak C (G)$ and $\mathfrak {C}^*(G)$ via the closed twin equivalence relation.

\begin{lemma}\label{tilde}
Let $G$ be a group.  Then the map $Z(g) \mapsto [g]$ is a graph isomorphism from $\Gamma_{\mathcal Z} (G)$ to $\mathfrak C(G)/{\sim}$ or $\mathfrak {C}^* (G)/{\sim}$.
\end{lemma}

\begin{proof}
Consider $g \in G \setminus Z(G)$ and view $g$ as a vertex in $\mathfrak C (G)$.  Note that $x \in G \setminus Z(G)$ will be a neighbor of $g$ if and only if $x \in C_G (g)$.  It follows that $C_G (g) \setminus Z(G) = N(g) \cup \{ g \}$.  Now when $h \in G \setminus Z(G)$, then $g {\sim} h$ as vertices in $C(G)$ if and only if $C_G (g) = C_G (h)$.  We can define a map from the vertices of $\mathfrak C(G)/{\sim}$ to the vertices of $\Gamma_{\mathcal Z} (G)$ by $[g] \mapsto Z(g)$ and by Corollary \ref{lem:corres}, this map will be a bijection.  Suppose $[g]$ and $[h]$ are two distinct vertices in $\mathfrak C(G)/{\sim}$.  We know there is an edge between these vertices if and only if $gh = hg$.  This occurs if and only if $h \in C_G (g)$.  On the other hand, there is an edge between $Z(g)$ and $Z(h)$ if and only if $h \in C_G(g)$.  Thus, the function we have defined will be a bijection on edges, and so, we have the graph isomorphism we desire from $\mathfrak{C} (G)/{\sim}$ to $ \Gamma_{\mathcal{Z}} (G)$.  

Let $a_1, \dots, a_n$ be a transversal for $Z(G)$ in $G$.  We see that $a_i$ and $a_j$ are neighbors if and only if $a_j \in C_G (a_i)$.  Also, $a_i {\sim} a_j$ if and only if $C_G (a_i) = C_G(a_j)$.  Hence, the rest of the argument for $\mathfrak {C}^* (G)$ is similar enough to the argument for $\mathfrak{C} (G)$ that we do not write it here.
%We note that if we restrict ourselves to a coset of $Z (G)$ in $G$, this still works.  The argument for $\mathfrak {C}^* (G)$ is similar enough that we do not write it here.
\end{proof}

We next have a purely graph theoretic result about quotient graphs.  

\begin{lemma}\label{bijection}
Let $\Gamma$ be a graph.  Then there is a bijection between the connected components of $\Gamma$ and the connected components of $\Gamma/{\sim}$.  Connected components of diameter $1$ in $\Gamma$ get mapped to isolated vertices in $\Gamma/{\sim}$.  Otherwise, this bijection preserves the diameter of the connected components.
\end{lemma}

\begin{proof}
Suppose $u$ and $v$ are vertices in $\Gamma$ and use $\leftrightarrow$ to denote edges in either $\Gamma$ or $\Gamma/\sim$.  Then it is not difficult to see that $u = u_0 \leftrightarrow u_1 \leftrightarrow \dots \leftrightarrow u_n = v$ is a path in $\Gamma$ if and only if $[u] = [u_0] \leftrightarrow [u_1] \leftrightarrow \dots \leftrightarrow [u_n] = [v]$ is a path in $\Gamma/\sim$.  Thus, we will have a bijection between connected components of the two graphs.  If a connected component of $\Gamma$ is a complete graph, then the corresponding connected component in $\Gamma/\sim$ will be an isolated vertex.  Otherwise, there will be two points whose distance is at least two, and the vertices in the shortest path between the two points will yield distinct equivalence classes, so the shortest path in the quotient will have the same length. 	
\end{proof}

We now apply this lemma.

\begin{corollary} \label{replacement}
Let $G$ be a group.  Then the following are true:
\begin{enumerate}
\item $\mathfrak C(G)/{\sim} \cong \mathfrak {C}^* (G)/{\sim}$.
\item There is a bijection between the connected components of $\mathfrak{C} (G)$ and $\Gamma_{\mathcal Z} (G)$.  This bijection maps connected components that are complete graphs to isolated vertices and preserves the diameters of all other connected components.  In particular, any connected component of $\Gamma_{\mathcal{Z}} (G)$ that is complete (i.e. is a clique) will be an isolated vertex.
\end{enumerate}
\end{corollary}

\begin{proof}
This is an immediate consequence of Lemmas \ref{bijection} and \ref{tilde}.
\end{proof}

Because of Corollary \ref{replacement}, we can use $\Gamma_{\mathcal{Z}} (G)$ when proving results regarding $\mathfrak{C} (G)$.  In particular, in the Introduction, we stated most of the results in terms of $\mathfrak{C} (G)$, but we will actually prove them for $\Gamma_{\mathcal{Z}} (G)$.

Observe that if $G$ and $H$ are groups such that $\mathfrak {C} (G) \cong \mathfrak {C} (H)$, then $\mathfrak C(G)/{\sim} \cong\mathfrak C(H)/{\sim}$.  Also, by Lemma \ref{tilde}, we see that $\Gamma_{\mathcal Z} (G) \cong \Gamma_{\mathcal Z} (H)$.  We next will show that if $G$ and $H$ are isoclinic then $\Gamma_{\mathcal Z} (G) \cong \Gamma_{\mathcal Z} (H)$.  However, isoclinism need not preserve order.  

On the other hand, if $\mathfrak C(G) \cong \mathfrak C(H)$, then $|G| - |Z(G)| = |H| - |Z(H)|$, and so, 
$$|Z(G)| (|G:Z(G)| - 1) = |Z(H)| (|H:Z(H)| - 1).$$  When $G$ and $H$ are isoclinic, we have $|G:Z(G)| = |H:Z(H)|$, and so, $|G:Z(G)| - 1 = |H:Z(H)| - 1$.  If $G$ and $H$ are isoclinic and $\mathfrak C (G) \cong \mathfrak C(H)$, we see that $|Z(G)| = |Z(H)|$.  This implies that $|G| = |H|$.  In particular, we conclude that if $G$ and $H$ are isoclinic but do not have the same order, then $\mathfrak C(G) \not\cong \mathfrak C(H)$.  Hence, in such a case, we have $\Gamma_{\mathcal Z} (G) \cong \Gamma_{\mathcal Z} (H)$ does not imply that $\mathfrak C (G)$ and $\mathfrak C(H)$ are isomorphic.  %However, it is an open question does $\Gamma_{\mathcal Z} (G) \cong \Gamma_{\mathcal Z} (H)$ and $|G| = |H|$ imply that $\mathfrak C(G) \cong \mathfrak C(H)$?

\begin{lemma} \label{isocl 1}
If $G$ and $H$ are isoclinic groups, then $\Gamma_{\mathcal Z} (G) \cong \Gamma_{\mathcal Z} (H)$.
\end{lemma}

\begin{proof}
We have seen that if $G_1$ and $G_2$ are isoclinic, then $\mathfrak {C}^*(G_1) \cong \mathfrak {C}^* (G_2)$.  It follows that ${\mathfrak {C}^* (G_1)}/{\sim} \cong {\mathfrak {C}^*(G_2)}/{\sim}$.  Applying Lemma \ref{tilde}, we obtain the result.
\end{proof}

Note that the converse of Lemma \ref{isocl 1} is not true.  In particular, the counterexamples mentioned for the converse of Theorem \ref{intro thm 1} will be counterexamples here also.  We also give examples here where $G$ and $H$ have isomorphic centralizer graphs but where $|G| = |G:Z(G)|$ and $|H| = |H:Z(H)|$ are different, so they are not isoclinic.  Let $p$ be a prime and let $a$ be a positive integer.  Suppose that $q$ and $r$ are distinct primes that divide $p^a - 1$.  Take $F$ to be a field of order $p^a$, and take $A$ to be the additive group of $F$ and take $Q$ to be the subgroup of order $q$ and $R$ to be the subgroup of order $r$.  Set $G_1$ to be the semi-direct product of $Q$ acting on $A$ and $G_2$ to be the semi-direct product of $R$ acting on $A$.  It is not difficult to see that $G_1$ and $G_2$ are both Frobenius groups and both of their centralizer graphs will consist of $p^a + 1$ isolated vertices.  On the other hand, the commuting graph of $G_1$ will have one connected component that is a complete graph of size $p^a-1$ and $p^a$ connected components that are complete graphs of size $q-1$, while the commuting graph of $G_2$ will have one connected component that is a complete graph of size $p^a-1$ and $p^a$ connected components that are complete graphs of size $r-1$.  

We continue with $p,a,q,r$ as above.  Let $G_3$ be the semi-direct product of $QR$ acting on $A$.  Again, $G_3$ will be a Frobenius group and its centralizer graph consists of $p^a + 1$ isolated vertices.  Its commuting graph will have a connected component that is a complete graph of size $p^a-1$ and $p^a$ connected components that are complete graphs of size $qr - 1$.  On the other hand, take $H = G_1 \times R$ so that $|H| = |G_3|$.  It is not difficult to see that the centralizer graph of $H$ is isomorphic to the centralizer graph of $G_1$, and hence to that of $G_3$.  And since $|H:Z(H)| = |G_1:Z(G_1)| = p^aq$ is different than $|G_3:Z(G_3)| =p^a qr$, we see that $H$ and $G_3$ have non-isomorphic commuting graphs.  This shows that we can have $\Gamma_{\mathcal{Z}} (G) \cong \Gamma_{\mathcal {Z}} (H)$ and $|G| = |H|$, but $\mathfrak{C} (G) \not\cong \mathfrak{C} (H)$.

%Is it possible the converse of Lemma \ref{isocl 1} is true? **(No, but maybe for $p$-groups?)**  In particular, we would like to see examples of groups $G$ and $H$ where $G$ and $H$ are not isoclinic, but $\Gamma_{\mathcal Z} (G) \cong \Gamma_{\mathcal Z} (H)$. 

%**There is an example for this in Frobenius groups.**

\section{Connectivity}

We now explore the connectivity of the graph $\Gamma_{\mathcal{Z}} (G)$.  We begin to explore the connection between the diameter of the centralizer graph in terms of $|G:Z(G)|$.

\begin{lemma} \label{intersection cent}
Let $G$ be a group, and $g,h \in G \setminus Z(G)$.  Then the following are equivalent:
\begin{enumerate}
\item $C_G (g) \cap C_G (h) > Z (G)$.
\item $Z(g)$ and $Z (h)$ are connected and have distance at most $2$ in $\Gamma_{\mathcal Z} (G)$.
\item $g$ and $h$ have distance at most $2$ in $\mathfrak {C} (G)$.
\item $gZ(G)$ and $h Z(G)$ have distance at most $2$ in $\mathfrak{C}^* (G)$.
\end{enumerate}
\end{lemma}

\begin{proof}
It is easy to see that the last three conditions are equivalent.  Thus, we must show that the first two are equivalent.  Suppose that $Z(G) < C_G (g) \cap C_G (h)$.  Since $Z(G) < C_G (g) \cap C_G (h)$, we can find an element $k \in (C_G (g) \cap C_G (h)) \setminus Z(G)$.  Observe that $Z(g)$ and $Z(h)$ are both adjacent to $Z(k)$ in $\Gamma_{\mathcal {Z}} (G)$.  Conversely, suppose that $Z(g)$ and $Z (h)$ are connected and have distance at most $2$ in $\Gamma_{\mathcal{Z}} (G)$.  If $Z(g)$ and $Z(h)$ are adjacent, then $Z(h) \le C_G (h)$, and since $Z(h) \le C_G (h)$, we have $Z(G) < Z(h) \le C_G (g) \cap C_G (h)$.  On the other hand, if they have a distance $2$, then $Z(k)$ is adjacent to both of them.  We have $Z(k)$ contained in both $C_G (g)$ and $C_G (h)$, so $Z(G) < Z(k) \le C_G (g) \cap C_G (h)$, and we have condition (1).
\end{proof}

We next show that if $|G'|$ is ``small'' compared to $|G:Z(G)|$, then the centralizer graph is connected and has diameter at most $2$.  Note that this yields Theorem \ref{intro thm 2} from the Introduction in the context of $\mathfrak{C} (G)$.

\begin{theorem} \label{diameter sq rt}
Let $G$ be a group.  If $|G'| < |G:Z(G)|^{1/2}$, then $\Gamma_{\mathcal Z} (G)$ is connected and has diameter at most $2$.
\end{theorem}

\begin{proof}
Let $g \in G \setminus Z(G)$.  Observe that the conjugacy class of $g$ is a subset of $gG'$.  Since the size of the conjugacy class of $g$ is $|G:C_G (g)|$, we have $|G:C_G (g)| \le |gG'| = |G'|$.  By Lemma \ref{intersection cent}, if either $G$ is disconnected or the diameter of $\Gamma_{\mathcal Z} (G)$ is greater than $2$, then there exist elements $g, h \in G \setminus Z(G)$ such that $C_G(g) \cap C_G (h) = Z(G)$.  Note that the set $C_G (g) C_G (h)$ may not be a subgroup of $G$, but still satisfies $|C_G (g) C_G (h)| = |C_G (g)| |C_G (h)| / |Z(G)|$.  Thus, 
$$|C_G (g)|/|Z(G)| = |C_G (g) C_G (h)|/|C_G(h)| \le |G|/|C_G (h)|.$$  
We see that 
$$\frac {|G|}{|Z(G)|} = \frac {|G|}{|C_G (g)|}\frac {|C_G (g)|}{|Z(G)|} \le \frac {|G|}{|C_G (g)|}\frac{|G|}{|C_G (h)|} \le |G'|^2.$$  
Since this contradicts $|G'| < |G:Z(G)|^{1/2}$, we obtain the desired conclusion.
%We have $$ \frac {|C_G (g) C_G (h)|}{|Z(G)|} = \frac {|C_G(g)||C_G(h)|}{|Z(G)|} \frac 1{|Z(G)} =   \frac{|C_G (g)|}{|Z(g)|} \frac {|C_G (h)|}{|Z(G)|}.$$   \ge |G'| |G'| = |G'|^2.$$  
\end{proof}

We now have the following characterization of sets that are unions of connected components in $\Gamma_{\mathcal Z} (G)$. 

\begin{lemma} \label{graph comp}
Let $G$ be a group, let $A = \{ Z_1, \dots, Z_n \} \subseteq {\mathcal Z} (G)$ be a subset of $\Gamma_{\mathcal Z} (G)$, and let $C_i = C_G (Z_i)$ for $i = 1, \dots, n$.  Then $A$ is a union of connected components of $\Gamma_{\mathcal Z} (G)$ if and only if $\cup_{i=1}^n Z_i = \cup_{i=1}^n C_i$.
\end{lemma}

\begin{proof}
Since $Z_i \le C_i$, we have $\cup_{i=1}^n Z_i \subseteq \cup_{i=1}^n C_i$.  Observe that $\cup_{i=1}^n C_i \setminus Z(G)$ consists of all of the neighbors of the vertices $Z_1, \dots Z_n$.  Now, $A$ is a union of connected components if and only if it contains all of the neighbors of its elements.  Hence, $A$ is a union of connected components if and only if $\cup_{i=1}^n Z_i$ contains all of the neighbors of $Z_1, \dots, Z_n$.  The conclusion now follows.
\end{proof}

When $\Gamma$ is a graph and $A$ is a subset of the vertices of $\Gamma$, we say $A$ is a {\it connected subset} if there is a path between every pair of vertices in $A$.  The following result follows immediately from the previous lemma.

\begin{corollary} \label{connected comp union}
Let $G$ be a group, let $A = \{ Z_1, \dots, Z_n \} \subseteq {\mathcal Z} (G)$, and let $C_i = C_G (Z_i)$ for $i = 1, \dots, n$.  If $A$ is a connected subset of  $\Gamma_{\mathcal Z} (G)$, then $A$ is a connected component of $\Gamma_{\mathcal Z} (G)$ if and only $\cup_{i=1}^n Z_i = \cup_{i=1}^n C_i$.  In particular, if $\Xi$ is a connected component of $\Gamma_{\mathcal{Z}} (G)$, then $\displaystyle \bigcup_{Z(x) \in \Xi} Z(x) = \bigcup_{Z(x) \in \Xi} C_G (x)$.
\end{corollary}

%\begin{proof}
%Notice that if $A$ is a union of connected components, then this follows from Lemma \ref{graph comp}.  Conversely, suppose $\cup_{i=1}^n Z_i = \cup_{i=1}^n C_i$.  Let $B$ be connected component containing $Z_1$.  Suppose $B$ is not contained in $A$.  Then there exists $Z \in B \setminus A$ and $Z_i \in B \cap A$ such th
%\end{proof}

Recall that a subset $C$ of the vertices of a graph $\Gamma$ forms a clique if the induced subgraph is a complete graph.  We now see that maximal cliques in $\Gamma_{\mathcal Z} (G)$ correspond to maximal abelian subgroups of $G$.

\begin{lemma} \label{graph 1}
Let $G$ be a group.  Then ${\mathcal A} \subseteq {\mathcal Z} (G)$ forms a clique in $\Gamma_{\mathcal Z}  (G)$ if and only if $\langle Z \in {\mathcal A} \rangle$ is an abelian subgroup of $G$.  Furthermore, ${\mathcal A}$ forms a maximal clique if and only if $\cup_{Z \in {\mathcal A}} Z$ is a maximal abelian subgroup of $G$.
\end{lemma}

\begin{proof}
Let $A = \langle Z \in {\mathcal A} \rangle$.  Assume that $A$ is abelian.  Then $A \le \cup_{Z \in {\mathcal A}} C_G (Z)$.  If $Z_1, Z_2 \in {\mathcal A}$, then $Z_1, Z_2 \le A$.  This implies $Z_2 \le C_G (Z_1)$, and so there is an edge between $Z_1$ and $Z_2$.  This implies that ${\mathcal A}$ forms a clique in $\Gamma_{\mathcal Z} (G)$.  Conversely, suppose that ${\mathcal A}$ forms a clique.  Consider $Y \in {\mathcal A}$.  Then for all $Z \in {\mathcal A}$, we have $Y \le C_G (Z)$.  This implies that $Y$ centralizes $Z$.  Now, $Y$ centralizes all the generators of $A$, so $Y \le Z (A)$.  Since $Y$ was arbitrary, we obtain $A = Z(A)$, and so, $A$ is abelian.

Suppose now that ${\mathcal A}$ forms a maximal clique.  By the first part, we see that $A = \langle Z \in {\mathcal A} \rangle$ is an abelian subgroup of $G$.  Let ${\mathcal B} = \{ Z (a) \mid a \in A \}$ and observe that $A = \cup_{Z \in {\mathcal B}} Z$.  By the first part, ${\mathcal B}$ is a clique, and it easy to see that ${\mathcal A} \subseteq {\mathcal B}$.  By maximality, we obtain ${\mathcal A} = {\mathcal B}$ which implies $A = \cup_{Z \in {\mathcal A}} Z$.  Finally, if $A$ is not a maximal abelian subgroup, then $A < A_1$ where $A_1$ is abelian.  Thus, there exists $b \in A_1 \setminus A$.  By the first part, ${\mathcal A} \cup \{ Z(b) \}$ will be a clique and this violates the maximality of ${\mathcal A}$.  Thus, $A$ is a maximal abelian subgroup.  Conversely, if $\cup_{Z \in {\mathcal A}} Z$ is a maximal abelian subgroup, then $A = \cup_{Z \in {\mathcal A}} Z$.  By the first part, ${\mathcal A}$ is a clique.  If it is not a maximal clique, then ${\mathcal A}$ is properly contained in the clique ${\mathcal B}$.  Since $A = \cup_{Z \in {\mathcal A}} Z$, any element $Y \in {\mathcal B} \setminus {\mathcal A}$ will not be a subgroup of $A$.  Thus, $A < \langle Y \mid Y \in {\mathcal B} \rangle$ and by the first part, $\langle Y \mid Y \in {\mathcal B} \rangle$ is abelian.  This contradicts the maximality of $A$, and proves the result.
\end{proof}

%%
%%%%%%%%%%%%%%%%%%%%%%%%%%%%%%%%%%%%%%%%%%%%

\section{Subordinate and independent vertices}

Let $\Gamma$ be a graph.  Let $u$ and $v$ be vertices of $\Gamma$.  We will say that $u$ is {\it subordinate} to $v$ if (1) $u$ and $v$ are neighbors, (2) every neighbor of $u$ is a neighbor of $v$, and (3) $v$ has a neighbor that is not a neighbor of $u$.  We will say $u$ is an {\it independent vertex} if it is not subordinate to any vertex.

\begin{lemma} \label{sub1}
Let $G$ be a group.  Let $g, h \in G \setminus Z(G)$.  Then the following are equivalent:
\begin{enumerate}
\item $Z (g) < Z(h)$.
\item $C_G (h) < C_G (g)$.
\item $h$ is a subordinate vertex to $g$ in ${\mathfrak C} (G)$.
\item $Z (h)$ is a subordinate vertex to $Z (g)$ in $\Gamma_{\mathcal Z} (G)$.
\end{enumerate}
\end{lemma}

\begin{proof}
By Lemma \ref{star 1}, we have $Z(g) < Z(h)$ if and only if $C_G (g) > C_G (h)$.

Suppose that $C_G (h) < C_G (g)$.  Since $h \in C_G (h) \le C_G (g)$, it follows that $g$ and $h$ commute, so $g$ and $h$ are neighbors in ${\mathfrak C} (G)$.  Suppose $k$ is a neighbor to $h$ in ${\mathfrak C} (G)$.  This implies that $k$ and $h$ commute, so $k \in C_G (h)$.  Since $C_G(h) \leq C_G (g)$, we see that $k$ and $g$ commute, so $g$ and $k$ are neighbors in ${\mathfrak C} (G)$.  Hence, every neighbor of $h$ is a neighbor of $g$.  Since $C_G (h) < C_G (g)$, we can find $x \in C_G (g) \setminus C_G (h)$.  It follows that $x$ and $g$ commute and $x$ and $h$ do not commute.  Thus, $x$ is a neighbor of $g$ that is not a neighbor of $h$.  This proves (3).

Suppose $h$ is a subordinate vertex to $g$ in ${\mathfrak C} (G)$.  Since $g$ and $h$ are neighbors in ${\mathfrak C} (G)$, we see that $[g]$ and $[h]$ are neighbors in ${\mathfrak C} (G)/{\sim}$.  Suppose $[x]$ is a neighbor of $[h]$ in ${\mathfrak C} (G)/{\sim}$.  It follows that $x$ is a neighbor of $h$ in ${\mathfrak C} (G)$, and so $x$ is a neighbor of $g$ in ${\mathfrak C} (G)$ and $[x]$ and $[g]$ are neighbors in ${\mathfrak C} (G)/{\sim}$.  Finally, we have a neighbor $y$ of $g$ in ${\mathfrak C} (G)$ that is not a neighbor of $[h]$.  It follows that $[y]$ is a neighbor of $[g]$ that is not a neighbor of $[h]$ in ${\mathfrak C} (G)/{\sim}$.  Thus, $[h]$ is a subordinate vertex to $[g]$ in ${\mathfrak C} (G)/{\sim}$.  Using Lemma \ref{tilde}, we see that this implies that $Z(h)$ is a subordinate vertex to $Z(g)$ in $\Gamma_{\mathcal Z} (G)$.

Suppose that $Z(h)$ is a subordinate vertex to $Z(g)$ in $\Gamma_{\mathcal Z} (G)$.  Since $Z(h)$ and $Z(g)$ are neighbors, we have $h \in C_G (g)$.  Let $x \in C_G (h) \setminus Z(G)$.  This implies that $Z(x)$ and $Z(h)$ are neighbors.  Since $Z(h)$ is subordinate to $Z(g)$, we have $Z(x)$ and $Z(g)$ are neighbors.  Hence, $x \in C_G (g)$.  We conclude that $C_G (h) \le C_G (g)$.  Finally, we have $y \in G \setminus Z(G)$ so that $Z(y)$ is a neighbor of $Z(g)$, but not a neighbor of $Z(h)$.  It follows that $y \in C_G (g)$, but $y \not\in C_G (h)$.  Thus, $C_G (h) < C_G (g)$.  This implies (2).
\end{proof}

The following is an immediate consequence of Lemma \ref{sub1}.

\begin{corollary}\label{sub2}
Let $G$ be a group.  Let $g \in G \setminus Z(G)$.  Then $g$ is an independent vertex in ${\mathfrak C} (G)$ if and only if $Z (g)$ is an independent vertex in $\Gamma_{\mathcal Z} (G)$.
\end{corollary}

We next use Lemma \ref{sub1} to see that $g$ being an independent vertex in $\mathfrak{C} (G)$ is equivalent to $C_G (g)$ being maximal in $\mathcal{C} (G)$ and $Z(g)$ being minimal in $\mathcal{Z} (G)$.

\begin{corollary}\label{independent max min}
Let $G$ be a group and let $g \in G \setminus Z(G)$.  Then the following are equivalent:
\begin{enumerate}
\item $g$ is an independent vertex in $\mathfrak{C} (G)$.
\item $C_G (g)$ is an independent vertex in $\Gamma_{\mathcal{Z}} (G)$.
\item $C_G (g)$ is maximal in $\mathcal{C} (G)$.
\item $Z(g)$ is minimal in $\mathcal{Z} (G)$. 
\end{enumerate}
\end{corollary} 

Using Lemma \ref{sub1}, we see that we can determine from $\mathfrak{C} (G)$ whether or not $G$ is an $F$-group.

\begin{corollary} \label{F group}
Let $G$ be a group.  The following are equivalent:
\begin{enumerate}
\item $G$ is an $F$-group.
\item Every vertex in $\Gamma_{\mathcal Z} (G)$ is an independent vertex.
\item Every vertex in $\mathfrak{C} (G)$ is an independent vertex.
\end{enumerate}
\end{corollary}

We next show that if $C_G (g)$ is not minimal in $\mathcal{C} (G)$, then $C_G (g)$ is not abelian.

\begin{lemma}\label{sub3}
Let $G$ be a group.  Let $g, h \in G \setminus Z(G)$.  If $C_G (h) < C_G (g)$, then $C_G (g)$ is not abelian.
\end{lemma}

\begin{proof}
We have $h \in C_G (h) < C_G (g)$.  If $C_G (g)$ is abelian, then $C_G (g)$ centralizes $h$, and so, $C_G (g) \le C_G (h)$, a contradiction.  Thus, $C_G (g)$ is not abelian. 
\end{proof}

We now characterize the isolated vertices in $\Gamma_{\mathcal Z} (G)$.

\begin{lemma} \label{isolated}
Let $G$ be a group.  Let $g \in G \setminus Z(G)$.  Then the following are equivalent:
\begin{enumerate}
\item $C_G (g)$ is abelian and $g$ is an independent vertex in ${\mathfrak C} (G)$.
\item $C_G (g)$ is abelian and $Z (g)$ is an independent vertex in $\Gamma_{\mathcal Z} (G)$.
\item For all $h \in G \setminus Z(G)$, either $C_G (h) = C_G (g)$ or $C_G (h) \cap C_G (g) = Z(G)$.
\item $C_G (h) = C_G (g)$ for all $h \in C_G (g) \setminus Z(G)$.
\item $Z (h) = Z (g)$ for all $h \in C_G (g) \setminus Z(G)$.
\item $Z (h) = C_G (g)$ for all $h \in C_G (g) \setminus Z(G)$.
\item $Z (g)$ is an isolated vertex in $\Gamma_{\mathcal Z} (G)$.
	%\item $C_G (g) \setminus Z(G)$ is a connected component in ${\mathfrak C} (G)$.
\end{enumerate}
\end{lemma}

\begin{proof}
The fact that (1) and (2) are equivalent is a consequence of Corollary \ref{sub2}.  Suppose (2).  Consider $h \in G \setminus Z(G)$.  If $h \in C_G (g)$, then since $C_G (g)$ is abelian, we have $C_G (g) \le C_G (h)$, and since $C_G (g)$ is independent, we obtain by Corollary \ref{sub2} that $C_G (g) = C_G (h)$.  Suppose $h \in G \setminus C_G (g)$.  If $C_G (g) \cap C_G (h) = Z (G)$, then we are done.  Thus, we suppose that $C_G (g) \cap C_G (h) > Z(G)$.  Thus, we can find $k \in (C_G (g) \cap C_G (h)) \setminus Z(G)$.  Since $k \in C_G (g)$, we have $C_G (k) =
C_G (g)$.  On the other hand, we now have $h \in C_G (k) = C_G (g)$ which is a contradiction.  Thus, we must have $C_G (g) \cap C_G (h) = Z(g)$.  This proves (3).

Suppose (3).  Consider $h \in C_G (g) \setminus Z(G)$.  It follows that $h \in C_G (g) \cap C_G (h)$, so $C_G (g) \cap C_G (h) > Z(G)$.  This implies that $C_G (h) = C_G (g)$.  This proves (4).  Notice that (4) and (5) are equivalent by Lemma \ref{lem:corres}.  Suppose (5).  For any $h \in C_G (g) \setminus Z(G)$, we have $h \in Z(h) = Z(g) = Z(C_G (g))$.  Since $Z(G) \le Z(C_G (g))$, we conclude that $C_G (g) \le Z(g)$.  Hence, $C_G (g) = Z(g) = Z(h)$.  This proves (6).  Suppose (6).  Since $g \in C_G (g) \setminus Z (G)$, we have $C_G (g) = Z(g) = Z(C_G (g))$ which implies that $C_G (g)$ is abelian.  Suppose $C_G (g)$ is a subordinate vertex.  By Lemma \ref{sub1}, there is an element $h \in G \setminus Z(G)$ so that $C_G (g) < C_G (h)$.  We have $Z(g) \le C_G (h)$, and by Lemma \ref{three}, this implies that $Z(h) \le C_G (g)$.  By (6), this implies that $Z(h) = C_G (g)$. It follows that $C_G (g)$ centralizes $C_G (h)$, and so, $C_G (h)$ centralizes $C_G (g) = Z(g)$.  Since $C_G (Z(g)) = C_G (g)$, we deduce that $C_G (h) \le C_G (g)$.  We now obtain $C_G (h) = C_G (g)$, which contradicts $C_G (g)$ is a subordinate vertex.  Therefore, $C_G (g)$ is an independent vertex.  This proves (2). 

%Notice that the fact that (7) and (8) are equivalent is a consequence of Lemma \ref{tilde}.  
We now show that (3) implies (7) and (7) implies (4).  Notice that (3) implies for all $h \in G \setminus Z(G)$ that either $C_G (g) = C_G (h)$ or $C_G (g) \cap C_G (h) = Z(G)$.  This implies that if $C_G (h) \ne C_G (g)$, then $C_G (g)$ and $C_G (h)$ are not adjacent in $\Gamma_{\mathcal Z} (G)$.  Consequently, $C_G (g)$ is an isolated vertex in $\Gamma_{\mathcal Z} (G)$.  This is (7).  Conversely, suppose (7).  Let $h \in C_G (g) \setminus Z(G)$, then $G_G (g)$ and $C_G (h)$ would be adjacent if they were different.  Thus, we must have $C_G (h) = C_G (g)$.  This implies (4).
\end{proof}

Hence, $Z$ is an isolated vertex in $\Gamma_{\mathcal Z} (G)$ if and only if $C_G (Z)$ is abelian and maximal among the subgroups in $\mathcal C(G)$.   

Recall that an empty graph is a graph with no edges.  One consequence of Lemma \ref{isolated} is that if $\Gamma_{\mathcal Z} (G)$ is an empty graph, then $C_G (x)$ is abelian for all $x \in G \setminus Z(G)$.  %A group $G$ is called a CA-group if $C_G (x)$ is abelian for all $x \in G \setminus Z(G)$.  (Some authors call these AC-groups.)  It is not difficult to see that every CA-group will be an $F$-group.  However, there are many $F$-groups that are not CA-groups.  For $p^{3n}$ with $n \ge 3$ and $p$ any prime, there will be only one CA-group that is ultraspecial, but many, many ultraspecial groups that are not CA-groups.
%
%We claim that if 
We next show that $G$ is a CA-group if and only if $\Gamma_{\mathcal Z} (G)$ is empty.

\begin{corollary} \label{CA group graph}
Let $G$ be a group.  Then the following are equivalent:
\begin{enumerate}
\item $\Gamma_{\mathcal Z} (G)$ is an empty graph.
\item Every connected component of $\mathfrak{C} (G)$ is a complete graph.
\item $G$ is a CA-group.
\end{enumerate}  
\end{corollary}

\begin{proof}
We have seen that Lemma \ref{isolated} implies that if $\Gamma_{\mathcal Z} (G)$ is an empty graph, then $G$ is a CA-group.  Assume that $G$ is a CA-group.  Notice that $C_G (x)$ abelian implies that $C_G (x)$ is maximal among abelian subgroups of $G$, since if $A$ is an abelian subgroup of $G$ that contains $x$, then $A \le C_G (x)$.  Since all centralizers are abelian, this implies that $C_G (x)$ is maximal in $\mathcal C (G)$.  Applying Lemma \ref{isolated} again, we see that this implies that $Z (x)$ is an isolated vertex.  Since $x$ was arbitrary, we deduce that $\Gamma_{\mathcal Z} (G)$ is an empty graph.
\end{proof}

We continue with a result that holds for any group $G$.  We find another condition that guarantees an isolated vertex.

\begin{lemma} \label{isolated prime index}
Suppose $G$ is a group, $g \in G$, and $p$ is a prime.  If $|C_G(g) : Z(G)| = p$, then $C_G(g)$ is an isolated vertex of $\Gamma_{\mathcal Z} (G)$.
\end{lemma}

\begin{proof}
Note that $|C_G(g) : Z(G)| = p$ implies that $C_G(g)$ is abelian.  Hence, by Lemma \ref{isolated} it suffices to show that $C_G(g)$ is a maximal centralizer.  There can be no $C_G(g) < C_G(h) < G$, as this would give $Z(G) < Z(h) < Z(g)$, but $Z(G)$ has index $p$ in $Z(g) = C_G(g)$.
\end{proof}

%In all of the examples we have seen, all of the elements of $\Gamma_{\mathcal Z} (G)$ that are not isolated vertices form a connected component of $\Gamma_{\mathcal Z} (G)$.   We would like to see an example of a group $G$ where the nonisolated vertices form more than one connected component of $\Gamma_{\mathcal Z} (G)$.

In this next lemma, we consider another possible way to obtain a connected component of $\Gamma_{\mathcal Z} (G)$.  This should be viewed as a generalization of Lemma \ref{isolated}.  A group with such a connected component is SmallGroup(128,227) and at least 103 other groups of order $128$.

\begin{lemma} \label{independent vertex component}
Let $G$ be a group.  Let $g \in G \setminus Z(G)$.  The following are equivalent:
\begin{enumerate}
\item For all $h \in G \setminus Z(G)$, either $C_G (h) \le C_G (g)$ or $C_G (h) \cap C_G (g) = Z(G)$.	
\item $C_G (h) \le C_G (g)$ for all $h \in C_G (g) \setminus Z(G)$.
\item $Z (g) \le Z (h)$ for all $h \in C_G (g) \setminus Z(G)$.
\item $Z (g)$ is an independent vertex in $\Gamma_{\mathcal{Z}} (G)$ and $Z(g)$ and all of its subordinate vertices form a connected component of $\Gamma_{\mathcal Z} (G)$.
\item $C_G (g) \setminus Z(G)$ is a connected component in ${\mathfrak C} (G)$.
\end{enumerate}
\end{lemma}

\begin{proof}
Suppose (1).  If $h \in C_G (g) \setminus Z (G)$, then $C_G (g) \cap C_G (h) > Z (G)$, and so, we must have $C_G (h) \le C_G (g)$ which is (2).  Now suppose (2).  Consider $h \in C_G (g) \setminus Z(G)$.  We must have $C_G (h) \le C_G (g)$.  This implies that $Z(h) \ge Z(g)$ which is (3).  

Suppose (3).   Suppose $Z(h)$ is adjacent to $Z(g)$.  We have $h \in C_G (g) \setminus Z(G)$.  This implies that $Z(g) \le Z(h)$.  By Lemma \ref{sub1}, this implies that $Z(h)$ is subordinate to $Z(g)$.  We conclude that $Z (g)$ is an independent vertex.  Note that since $Z(h)$ is subordinate to $Z (g)$, every neighbor of $Z (h)$ is a neighbor of $Z (g)$.  It follows that the connected component containing $Z (g)$ will consist of $Z(g)$ and its neighbors which are its subordinate vertices.  This proves (4).

Suppose (4).  Let $h$ be a neighbor of $g$.  It follows that $Z(h)$ is a neighbor of $Z(g)$, and so, $Z(h)$ is subordinate to $Z(g)$.   It follows that $h$ is subordinate to $g$.  Hence, all of the neighbors of $h$ are neighbors of $g$.  We conclude that every vertex of the connected component of $g$ is a neighbor of $g$, and so, the connected component of $g$ is the set $C_G (g) \setminus Z(G)$.  This proves (5).

Suppose (5).  Assume there exists an element $h \in G \setminus Z(G)$ so that $C_G (g) \cap C_G (h) > Z(g)$.  By Lemma \ref{intersection cent}, this implies that $h$ and $g$ are in the same connected component.  Hence, $h \in C_G (g)$. Suppose $k \in C_G (h) \setminus Z(G)$.  It follows that $k$ is adjacent to $h$, and so, $k$ is in the same connected component as $g$, and this implies that $k \in C_G (g)$.  This shows that $C_G (h) \le C_G (g)$ which proves (1). 
\end{proof}

\begin{corollary} \label{ind diam}
Let $G$ be a group and let $\Xi$ be a connected component of $\Gamma_{\mathcal{Z}} (G)$.  If there is an element $g \in G \setminus Z(G)$ so that $Z(g)$ is an independent vertex in $\Gamma_{\mathcal{Z}} (G)$ that is not an isolated vertex and $Z(g)$ and all of its subordinate vertices make up $\Xi$, then $\Xi$ has diameter $2$. 
\end{corollary}

\begin{proof}
Suppose first that there is an element $g \in G \setminus Z (G)$ so that $Z(g)$ is not an isolated vertex and $\Xi$ is made up of $Z (g)$ and its subordinate vertices.  Thus, if $Z(h) \in \Xi$, then $Z(g) \le Z(h)$, so $Z(h) \le C_G (h) \le C_G (g)$, and so, $Z(h)$ is adjacent to $Z(g)$.  This shows that every element in $\Xi$ is adjacent to $Z(g)$, and thus, $\Xi$ has diameter at most $2$.  Notice that if $\Xi$ has diameter $1$ then $C_G (g)$ is abelian, and this contradicts Lemma \ref{sub3}.  Hence, $\Xi$ has diameter $2$.
\end{proof}

Notice that when $\Xi$ is an isolated vertex $Z(g)$, the set $\displaystyle \bigcup_{Z(x) \in \Xi} = C_G (g)$ is a subgroup and when $\Xi$ is the independent vertex $Z(g)$ and its subordinate vertices, the set $\displaystyle \bigcup_{Z(x) \in \Xi} C_G(x) = C_G (g)$ is also a subgroup.  The question arises, can one have any other cases where $\displaystyle \bigcup_{Z(x) \in \Xi} C_G(x)$ is a subgroup?  

Suppose now that $G$ is a nonabelian $p$-group.  We obtain an easy condition which implies that a centralizer is abelian.  

\begin{lemma} \label{abel cent index}
Let $p$ be a prime, and let $G$ be a $p$-group.  If $x \in G \setminus Z(G)$ satisfies $|C_G (x):Z(G)| \le p^2$, then $C_G (x)$ is abelian.
\end{lemma}

\begin{proof}
We begin by observing that $x \in Z(C_G (x))$.  It follows that $|C_G (x):Z(C_G (x))| \le p$.  This implies that $C_G (x)$ is central-by-cyclic, and it is well-known that this implies that $C_G(x)$ is abelian.
\end{proof}

We now show that $p$-groups where $|G:Z(G)| \le p^3$ are CA-groups.  Notice that this gives Theorem \ref{intro thm 3} (1) in the context of $\mathfrak{C} (G)$.

\begin{lemma} \label{CA group}
Let $p$ be a prime, and suppose that $G$ is a $p$-group.  If $|G:Z(G)| \le p^3$, then $G$ is a CA-group.
\end{lemma}

\begin{proof}
Consider $x \in G \setminus Z(G)$.  We know that $C_G (x) < G$, so $|C_G (x):Z(G)| < |G:Z(G)| = p^3$.  Applying Lemma \ref{abel cent index}, we see that $C_G (x)$ is abelian.  Since $x$ was arbitrary, we conclude that $G$ is a CA-group.
\end{proof}

We now obtain a bound on the diameter of $\Gamma_{\mathcal{Z}} (G)$ in terms of $|G:Z(G)|$.  We mention that this is Theorem \ref{intro thm 3} (2) for $\mathfrak{C} (G)$.

\begin{theorem}\label{Z(G) le p^4 or p^5}
Let $p$ be a prime, and let $G$ be a $p$-group.  If $|G:Z(G)|$ is $p^4$ or $p^5$, then either $G$ is a CA-group or else the set of nonisolated vertices in $\Gamma_{\mathcal Z} (G)$ is connected and has diameter at most $4$.
\end{theorem}

\begin{proof}
If $G$ is a CA-group, then the conclusion holds.  Thus, we assume that $G$ is not a CA-group. Consider $x \in G \setminus Z(G)$ so that $C_G(x)$ is not abelian.  By Lemma \ref{abel cent index}, this implies that $|C_G (x):Z(G)| \ge p^3$, so $|G:C_G (x)|$ is $p$ if $|Z:Z(G)| = p^4$ and either $p$ or $p^2$ when $|G:Z(G)| = p^5$.  Let $u$ and $v$ be elements of $G \setminus Z(G)$ so that $Z(u)$ and $Z(v)$ are not isolated vertices of $G$.  

If $C_G (u)$ is not abelian, we take $u = s$.  If $C_G (u)$ is abelian, then since we are assuming that $Z(u)$ is not an isolated vertex, we may apply Lemma \ref{isolated} to see that $Z(g)$ is a subordinate vertex in $\Gamma_{\mathcal Z} (G)$.  Using Lemma \ref{sub1}, there exists $s \in G$ so that $C_G (u) < C_G (s)$.  Observe that $C_G (s)$ is not abelian and $Z(u)$ and $Z(s)$ are adjacent in $\Gamma_{\mathcal Z} (G)$. Similarly, we can find an element $t \in G \setminus Z(G)$ so that $C_G (t)$ is nonabelian and either $t = v$ or $Z(t)$ and $Z(v)$ are adjacent vertices in $\Gamma_{\mathcal Z} (G)$.  If $|G:Z(G)| = p^4$, observe that $|G:C_G (s)| = |G:C_G (t)| = p$ and if $|G:Z(G)| = p^5$, then $|G:C_G (s)|$ and $|G:C_G (t)|$ are either $p$ or $p^2$.  In all cases, we have $C_G (s) \cap C_G (t) > Z(G)$.  This implies the distance between $Z(s)$ and $Z(t)$ is at most $2$ by Lemma \ref{intersection cent}, and hence, the distance between $Z(u)$ and $Z(v)$ is at most $4$.  Since $u$ and $v$ were arbitrary, this proves the result.
\end{proof}

When $|G:Z(G)| = p^4$, we also are able to obtain information about the isolated vertices.

\begin{theorem}\label{Z(G) eq p^4}
Let $p$ be a prime and let $G$ be a $p$-group so that $|G:Z(G)| = p^4$ and $G$ is not a CA-group.  Then the following are true:
\begin{enumerate}
\item If $x \in G \setminus Z (G)$ is an element so that $Z (x)$ corresponds to an isolated vertex in $\Gamma_{\mathcal{Z}} (G)$, then $|C_G (x):Z(G)| = p$.
\item If the connected component of $\Gamma_{\mathcal{Z}} (G)$ consisting of nonisolated vertices has a unique independent vertex and does not consist of all of $G \setminus Z (G)$, then $\Gamma_{\mathcal{Z}} (G)$ will have $p^3$ isolated vertices.  
\item If the connected component of $\Gamma_{\mathcal{Z}} (G)$ consisting of nonisolated vertices has at least two independent vertices and does not consist of all of $G \setminus Z(G)$, then $\Gamma_{\mathcal{Z}} (G)$ will have at most $p^3 - p^2$ isolated vertices.
\end{enumerate}  
\end{theorem}

\begin{proof}
Since $G$ is not a CA-group there is an element $g \in G$ so that $C_G (g)$ is  not abelian.  By Lemma \ref{abel cent index}, we see that $|C_G (g):Z(G)| \ge p^3$.  Since $g \not\in Z(G)$, we have $|C_G (g):Z (G)| < p^4$, so $|C_G (g):Z(G)| = p^3$.  Suppose that $Z(x)$ corresponds to an isolated vertex in $\Gamma_{\mathcal{Z}} (G)$.  If $|C_G (x):Z(G)| \ge p^2$, then $|C_G (g) \cap C_G(x):Z(G)| \ge p$, and $Z(x)$ and $Z(g)$ are connected by Lemma \ref{intersection cent}, and this contradicts $Z(x)$ is isolated.  Thus, we have $|C_G (x):Z(G)| = p$.  This proves (1).

Let $\Xi$ be the connected component of $\Gamma_{\mathcal{Z}} (G)$ that consists of nonisolated vertices.  We suppose that $\Xi$ has a unique independent vertex $Z(g)$.  This implies $\Xi$ consists of $Z(g)$ and subordinate vertices for $Z(g)$.  By Lemma \ref{independent vertex component}, every vertex in $\Xi$ is a subgroup of $C_G (g)$.   Thus, the isolated vertices in $\Gamma_{\mathcal{Z}} (G)$ consist of $Z(u)$ where $u \in G \setminus C_G (x)$.  Using (1), we see that we obtain $(|G| - |C_G (g)|)/(|Z(u)| -|Z(G)|)$ isolated vertices.
$$\frac{|G| - |C_G(g)|}{|Z(u)| - |Z(G)|} = \frac{p^4|Z(G)| - p^3 |Z(G)|}{p|Z(G)| - |Z(G)|} = p^3.$$

Again, take $\Xi$ to be the connected component of $\Gamma_{\mathcal{Z}} (G)$ that consists of nonisolated vertices.  We now suppose that $\Xi$ has two independent vertices $Z(g)$ and $Z(h)$.  We see that $|C_G (h):Z(G)| = |C_G (h):Z(G)| = p^3$ and $|C_G(g) \cap C_G (h): Z(G)| \le p^2$.  Thus, $|(C_G (g) \cup C_G (h)) \setminus Z(G)| \ge (2p^3 - p^2) |Z(G)|$.  The number of isolated vertices is at most:
$$\frac{p^4|Z(G)| - (2p^3 - p^2) |Z(G)|}{p|Z(G)| - |Z(G)|} = \frac {p^4 -2p^3 +p^2}{p-1} = p^2 (p-1).$$
\end{proof}

We will provide an example later to show that when $G$ is a $p$-group and $|G:Z(G)| \ge p^6$, then $\Gamma_{\mathcal{Z}} (G)$ can have more than one connected component which is not an isolated vertex.  However, we wonder if in general, whether we can obtain a bound on the diameter of the connected components of $\Gamma_{\mathcal{Z}} (G)$ in terms of $|G:Z(G)|$.

When $G$ has nilpotence class at least $3$, we can extend this idea.  Let $G_1 = G$, and $G_{i+1} = [G_i,G]$ for $i \ge 1$, so that $G_2 = G'$, and $G_1, G_2, \dots$ are the terms of the lower central series of $G$.  We see that $G$ has nilpotence class $n$ if $G_n > 1$ and $G_{n+1} = 1$.  We note that this Theorem \ref{intro thm 4} for $\mathfrak{C} (G)$. 

\begin{theorem}\label{Fittingclass more than 3}
Let $p$ be a prime, and let $G$ be a $p$-group of nilpotence class $n \ge 3$.  If $|G:C_G (G_{n-1})|$ is either $p$ or $p^2$, then $\Gamma_{\mathcal{Z}} (G)$ has at most one connected component that is a nonisolated vertex and it has diameter at most $8$.
\end{theorem}

\begin{proof}
Since $n$ is the nilpotence class of $G$, we have $n \ge 3$. 
%$$G = G_0, G' = G_1, G_2, \dots, G_{n+1} = 1$$ 
%be the terms in the lower central series for $G$.  
Thus, $G_{n-1} \le G'$ and $G_{n-1} \not\le Z(G)$.  Take $C = C_G (G_{n-1})$ and note that $C < G$.  We know that $[G_{n-1},G_2] \le G_{(n-1) + 2} = G_{n+1} = 1$; so we also have $G_2 = G' \le C$.  We can find an element $y \in G_{n-1} \setminus Z(G)$.  Since $G_{n-1}$ is central in $C$, we have for every $x \in C \setminus Z(G)$, that $x$ centralizes $y$.  Thus, $Z (x)$ is adjacent to $Z (y)$ in $\Gamma_{\mathcal Z} (G)$.  It follows that the centers of the elements of $C \setminus Z(G)$ all lie in the same connected component $\Xi$ of $\Gamma_{\mathcal Z} (G)$.  Notice that the centers of the elements of $C \setminus Z(G)$ have distance at most $2$ from each other in this graph.  %%This is ok.

Suppose $g \in G \setminus C$.  If $C_G (g) \cap C > Z(G)$, then $Z(g)$ is connected by a path of distance at most $2$ to some vertex in $\Xi$, and thus, it has a distance at most $3$ to $Z (y)$.  We may assume that $C_G (g) \cap C = Z(G)$.  Since $G' \le C$, we see that $C_G(g)' \le C$.  Obviously, $C_G(g)' \le C_G (g)$, so $C_G (g)' \le C_C (g) \cap C = Z(G)$.  Now, $|C_G (g):Z (G)| \le p^2$.  Since $g \in Z (g) \setminus Z(G)$, we have $|C_G (g): Z(g)| \le p$, and so, $C_G (g)/Z(g)$ is central-by-cyclic.  Thus, $C_G (g)$ is abelian.  By Lemma \ref{isolated}, we have either $Z(g)$ is an isolated vertex or $Z(g)$ is a subordinate vertex.  If $Z(g)$ is a subordinate vertex, then there exists $h \in G \setminus Z(G)$ so that $Z(g)$ is subordinate to $Z(h)$.  By Lemma \ref{sub3}, we see that $C_G (h)$ is not abelian, and so, $C_G (h) \cap C > 1$.  Thus, $Z (g)$ is adjacent to $Z(h)$ and $Z(h)$ has distance at most $3$ to $Z(y)$.  We have now shown that every element in $\Xi$ has distance at most $4$ from $Z (y)$, so $\Xi$ has diameter at most $8$ and every vertex outside of $\Xi$ is an isolated vertex, proving the result.
\end{proof}

Working as in Theorem \ref{Z(G) eq p^4}, we can show that if $G$ has nilpotence class $n \ge 3$ and $|G:C_G(G_{n-1})| = p^4$, then the isolated vertices all have index $p$ over $Z(G)$ and the number of isolated vertices will be at most $|C_G (G_{n-1}):Z(G)|$.  We wonder when $G$ has nilpotence class $n \ge 3$ whether we can bound the diameter of the connected components of $\Gamma_{\mathcal{Z}} (G)$ in terms of $|G:C_G (G_{n-1})|$.

%\begin{lemma}
%Let $p$ be a prime, and let $G$ be a $p$-group of nilpotence class $n \ge 3$.  If $|%G:C_G (G_{n-1})| = p$ and $Z(x)$ is an isolated vertex in $\Gamma_{\mathcal{Z}} (G)$ %then $|C_G (x):Z(G)| = p$ and the number of isolated vertices is at most $p^3$.   
%\end{lemma}

%\begin{proof}

%\end{proof}
%We close this section by mentioning a generalization of CA-groups.  

%We close this section with the observation that a group $G$ is an $F$-group if and only if all of the vertices in $\Gamma_{\mathcal Z} (G)$ (or $\mathfrak C (G)$ are independent vertices; so the class of $F$-groups can be determined by its commuting graph.

%\section{Connected components}

%Let $\Gamma$ be a graph.  %Let $A$ be a subset of the vertices of $\Gamma$.  We will say that $A$ is a {\it connected set} if there is a path between any pair of points in $A$.  

\section{Examples}

In this section, we define a family of examples.  We will use this family of examples to obtain $p$-groups whose centralizer graphs have various structures of interest.   One such group is to construct a $p$-group $G$ where $\Gamma_{\mathcal Z} (G)$ has more than one nontrivial connected component.  In fact we do more.  We show that for any integer $k$ and integers $n_1,\dots,n_k$, we can construct a $p$-group $G$ of nilpotence class $2$ so that $\Gamma_{\mathcal Z} (G)$ has exactly $k$ nontrivial connected components, say $C_1$, \dots, $C_k$, and each component $C_i$ has diameter $n_i$.

%Recall that in a group $G$, if a connected component of $\Gamma_{\mathcal Z} (G)$ is complete (i.e. it is a clique), then it is an isolated vertex (see Corollary \ref{replacement}).  And so the connected
Components of $\Gamma_{\mathcal Z} (G)$ are either isolated vertices or have diameter at least $2$.  We say that a connected component of $\Gamma_{\mathcal Z} (G)$ is \textit{trivial} if it consists of an isolated vertex.

%For our first example,
\subsection{$G(p,n,S)$}

We fix a prime $p$ and a positive integer $n$.  We will take $S \subseteq \{ \{ i,j \} \,\, | \,\, 1 \leq i < j \leq n \}$.  Note that we may view $S$ as a graph on the set $\{ 1, \dots, n \}$.   We will write $\widebar{S}$ for the complement graph of $S$.  We define a group $G = G(p,n,S)$ as follows.  We define $G$ to be the $p$-group of nilpotence class $2$ on the formal variables $x_1, \dots, x_n$ and $z_{i,j}$ for all $\{ i,j \} \in S$, $i < j$ where the $z_{i,j}$'s lie in the center of $G$ and $x_i$ and $x_j$ commute if $\{i,j \} \in \widebar{S}$ and $[x_i,x_j]=z_{i,j}$ if $\{i,j \} \in S$.  When convenient, we may write $z_{j,i}$ for $(z_{i,j})^{-1}$.  The elements of $G$ consist of all of the formal products (or words) over the formal variables.  So $G$ is a $p$-group with $|G| = p^{n + |S|}$.  For our first observation, we see that $x_i$ is in the center of $G$ if and only if $i$ an isolated vertex of $S$ viewed as a graph. 

\begin{proposition}\label{prop:degen}
Let $G = G(p,n,S)$, and fix $i$, $1 \leq i \leq n$.  Then $x_i \in Z(G)$ if and only if $i$ is an isolated vertex in $S$ viewed as a graph.
\end{proposition}

\begin{proof}
If $i$ is an isolated vertex in $S$, then for every $j$ with $i \ne j$, we have $\{ i,j \} \notin S$, and thus $[x_i, x_j]=1$ for all $x_j$.  Hence $x_i \in Z(G)$.  On the other hand, if $x_i \in Z(G)$, then $[x_i, x_j]=1$ for all $j \ne i$ and so there are no edges in the graph $S$ that contain $i$, and so $i$ is an isolated vertex.
\end{proof}

%We wish to exclude this degenerate case, and so as a definition, we say that $G = G(n,p,S)$ is \textit{nondegenerate} provided that $$\bigcup S = \{ 1, \dots, n \}.$$

Thus, we will exclude graphs $S$ with isolated vertices so that 
$$Z(G (p,n,S)) = \langle z_{i,j} \mid \{i,j\} \in S \rangle.$$  Such $G(p,n,S)$ we call \textit{non-degenerate}.

We now analyze centralizers of elements of $G(p,n,S)$.  Of course for any group $G$ with $g \in G$ and $z \in Z(G)$, we have $C_G(g) = C_G(gz)$.  We first compute the centralizers of the $x_i$'s. 

\begin{lemma} \label{single}
Let $G = G(p,n,S)$, and fix $i$, $1 \leq i \leq n$.  Then $C_G (x_i) = \langle x_i, x_j \mid \{ i, j\} \in \widebar{S} \rangle Z(G)$.
\end{lemma}

\begin{proof}
We have $x_i \in C_G (x_i)$ and if $\{ i,j \} \not\in S$, then $x_j \in C_G (x_i)$.  Hence, $\langle x_i, x_j \mid \{ i, j\} \in \widebar{S} \rangle Z(G) \le C_G (x_i)$.  On the other hand, suppose $g \in C_G (x_i)$.  We know $\displaystyle g = (\prod_{t=1}^n x_t^{a_t}) z$ for some $z \in Z(G)$.  We see that $\displaystyle [x_i,g] = \prod z_{i,t}^{a_t}$ for the $t$'s such that $\{i,t \} \in S$.  Hence, this commutator will be $1$ if and only if $a_t = 0$ for all $t$ such that $\{i,t \} \in S$.  It follows that $g \in  \langle x_i, x_j \mid \{ i, j\} \in \widebar{S} \rangle Z(G)$.  This implies that $C_G (x_i) \le  \langle x_i, x_j \mid \{ i, j\} \in \widebar{S} \rangle Z(G)$, and we obtain the desired equality.
\end{proof}

We now see that $C_G (x_i)$ and $C_G (x_j)$ are equal if and only if $i$ and $j$ are true twins in $\widebar{S}$.  Since we will want all of the $C_G (x_i)$'s to give distinct vertices in $\Gamma_{\mathcal{Z}} (G)$, we will assume that the graph of $\widebar{S}$ has no true twin vertices.

\begin{corollary} \label{single 1}
Let $G = G(p,n,S)$, and fix $i \ne j$, $1 \leq i, j \leq n$.  Then $C_G (x_i) = C_G (x_j)$ if and only if $i$ and $j$ are true twins in $\widebar{S}$. 
\end{corollary}

Assuming they are not equal, $C_G (x_i)$ and $C_G (x_j)$ will be adjacent in $\Gamma_{\mathcal{Z}} (G)$ exactly when $i$ and $j$ are adjacent in $\widebar{S}$.

\begin{corollary} \label{single 2}
Let $G = G(p,n,S)$, and fix $i \ne j$, $1 \leq i, j \leq n$.  Then $C_G (x_i)$ and $C_G (x_j)$ are adjacent in $\Gamma_{\mathcal{Z}} (G)$ if and only if $i$ and $j$ are adjacent and not true twins in $\widebar{S}$. 
\end{corollary}

We now consider centralizers of products.  The elements of $G = G(p,n,S)$ have the form $\displaystyle g = \prod_{i=1}^n x_i^{a_i} z$ where $0 \le a_i \le p-1$ and $z \in Z(G)$.  We will say that $x_i$ is involved in $g$ if $a_i \ne 0$.  Note that this next lemma is a generalization of Lemma \ref{single}.  %Note that $D$ may not be a clique in $\widebar{S}$.

\begin{lemma}\label{clique}
Let $G = G(p,n,S)$.  Suppose $C \subseteq \{ 1, \dots, n \}$ forms a clique (of size at least $2$) in $\widebar{S}$.  Let $\displaystyle g = \prod_{i \in C} x_i^{a_i}$ with $a_i \ge 1$ for each $i \in C$.  Let $D$ be the set of all $j$ in $\{ 1, \dots, n\} \setminus C$ so that $j$ is adjacent in $\widebar{S}$ to every vertex in $C$.  Then $C_G (g) = \langle x_j \mid j \in C \cup D \rangle Z(G)$.  In fact, $C_G (g) \le C_G (x_i)$ for every $i \in C$.  
\end{lemma}

\begin{proof}
Since most of the proof of this is very similar to the proof of Lemma \ref{single}, we do not repeat it here.  We only comment on the last statement.  Notice that for $i \in C$, the $x_j$'s that commute with $x_i$ will include all the $x_h$ with $h \in C \cup D$, but may include other elements in $\{ 1, \dots, n \}$ which is why we have to allow containment in $C_G (g) \le C_G (x_i)$.
\end{proof}

We now consider products that contain elements of $\{ x_1, \dots, x_n \}$ that do not form a clique in $\widebar{S}$.  We start with products that contain elements of $\{ x_1, \dots, x_n \}$ whose distance in $\widebar{S}$ is at most $2$.

%and the elements of distance $2$ all have common neighbors that form a clique.

Given a set of vertices $D$ of a graph $\Gamma$, we say that $x \in D$ is a \textit{dominating vertex of $D$} if $x$ is adjacent in $\Gamma$ to every $y \in D$ with $x \neq y$.

Note that if $D$ is a set of vertices of a graph $\Gamma$, and if $E$ is the set of all dominating vertices of $D$, then $D = D' \cup E$ where $D' = D \setminus E$ and if $E$ is nonempty, then $E$ is a clique in $\Gamma$.

\begin{lemma} \label{type 2}
Let $G = G(p,n,S)$.  Suppose $C, D \subseteq \{ 1, \dots, n \}$ so that
\begin{enumerate}
\item $C$ contains at least two elements, 
\item all the elements in $C$ pairwise have distance $2$ in $\widebar{S}$.
%\item every element in $C$ is adjacent in $\widebar{S}$ to all the elements in $D \cup E$,
%\item $D$ is nonempty, $E$ may or may not be empty, $D$ and $E$ are disjoint.
\item $D$ is the set of all elements that are adjacent in $\widebar{S}$ to all the elements in $C$.
\item $D$ is nonempty.
%\item $E$ is the set of all elements in $\widebar{S}$ that are adjacent to all the elements in $C \cup D$. 
%\item If $E$ is nonempty, then $E$ is a clique in $\widebar{S}$.
\item Write $D = D' \cup E$ where $E$ is the set of all dominating vertices of $D$ in $\widebar{S}$ and $D' = D \setminus E$.
\end{enumerate} 
If $\displaystyle g = \prod_{i \in C} x_i^{a_i} \prod_{j \in E} x_j^{b_j}$ with $a_i \ge 1$ and $b_j \ge 0$, then 
$$C_G (g) = \langle \prod_{i \in C} x_i^{a_i}, x_j \mid j \in D \rangle Z(G).$$  
If $E$ is not empty, then $C_G (g)$ is a subordinate vertex in $\Gamma_{\mathcal{Z}} (G)$.   
\end{lemma}

\begin{proof}
We see that $g \in C_G (g)$.  Also, if $l \in D$, then $l$ is adjacent in $\widebar{S}$ to all $i \in C$ and $k \in E$ with $l \ne k$, so it follows that $x_l$ commutes with every component of $g$, and hence, $x_l$ commutes with $g$.  Thus, 
$$\langle g, x_j \mid j \in D \rangle  = \langle \prod_{i \in C} x_i^{a_i}, x_j \mid j \in D \rangle \le C_G (g).$$

Observe that if $h \in G$ involves some $x_l$ for $l \not\in C \cup D$, then $h$ will not commute with $g$ since there will be some $i \in C$ so that $i$ and $l$ are adjacent in $S$ and thus $x_i$ and $x_l$ do not commute.  Thus, the elements in $C_G (g)$ will only contain $x_l$ with $l \in C \cup D$.  Since we know that $C_G (g)$ contains all the $x_l$ for $l \in D$, it suffices to consider the elements that involve only $x_i$ with $i \in C$.  Computing commutators, one can show that such an element will commute with $g$ if and only if it is a power of $\displaystyle \prod_{i \in C} x_i^{a_i}$.  If $E$ is not empty, then $C_G (g) < C_G (x_h)$ for $h \in E$, and so, $C_G (g)$ is subordinate.  %If $E$ is empty, then we know that there is no element in $\{ 1, \dots , n\}$ that is adjacent to all of the $x_i$'s involved in $g$ and so there will be no $x_i$ so that $C_G (g)$ is contained in $C_G (x_i)$.	
\end{proof}

The remaining case is where we have an element $g$ and there is no $i$ so that $i$ is adjacent in $\widebar{S}$ to all of the $j$'s so that $x_j$ is involved in $g$.

\begin{lemma} \label{isolated type}
Let $G = G(p,n,S)$.   Suppose $C \subseteq \{ 1, \dots, n \}$ contains at least two elements and there is no element in $\{1 , \dots, n \}$ that is adjacent in $\widebar{S}$ to all of the elements in $C$.  If $\displaystyle g = \prod_{k=1}^n x_k^{a_k}$ and $a_i \geq 1$ when $i \in C$, then $C_G (g) = \langle g \rangle Z(G)$ and $C_G (g)$ is an isolated vertex in $\Gamma_{\mathcal{Z}} (G)$.  %Suppose that there exist $g_i$ and $g_j$ so that the distance between $i$ and $j$ in the complement of $S$ is at least $3$ (including $\infty$).  
\end{lemma}

\begin{proof}
We have $g \in C_G (g)$.  Let $h \in C_G (g)$, and suppose that $x_k$ is involved in $h$.  Assume first that $x_k$ is not involved in $g$.  Then there is some $i \in C$ so that $x_i$ and $x_k$ do not commute and so $g$ and $h$ do not commute.  Thus, $h$ involves only involves $x_i$ for $i \in C$.  Computing commutators, one can show that such an element will commute with $g$ if and only if it is a power of $\displaystyle \prod_{i \in C} x_i^{a_i}$.  By Lemma \ref{isolated prime index}, $C_G (g)$ is an isolated vertex in $\Gamma_{\mathcal{Z}} (G)$.
\end{proof}

We now consider independent vertices in $\widebar{S}$.  We also revisit Lemma \ref{type 2} where $E$ is empty.

\begin{lemma}
Let $G = G(p,n,S)$.  Assume that no $i$ is an isolated vertex in $S$ (viewed as a graph).  Then the following are true:
\begin{enumerate}
\item For each $1 \le i \le n$, we have $C_G (x_i)$ is an independent vertex of $\Gamma_{\mathcal{Z}} (G)$ if and only if $i$ is an independent vertex of $\widebar{S}$.  
\item If $g \in G$ satisfies the hypotheses of Lemma \ref{type 2} and $E$ is empty, then $C_G (g)$ is an isolated vertex in $\Gamma_{\mathcal{Z}} (G)$.
\end{enumerate}
\end{lemma}

\begin{proof}
Suppose $i$ is an independent vertex in $\widebar{S}$.  By Lemma \ref{clique}, we see that any $g$ where the $x_j$'s involved in $g$ all have $j$'s coming from a clique that includes $i$ will have $C_G (g) \le C_G (x_i)$.  If $j$ is a neighbor of $i$ in $\widebar{S}$ we will have $C_G (x_j) < C_G (x_i)$ when $j$ is a subordinate vertex to $i$, and we will have $C_G (x_j) = C_G(x_i)$ when $j$ and $i$ are true twins, and we will have $C_G (x_i)$ and $C_G (x_j)$ are not comparable when $j$ is a different independent vertex.  We see from Lemmas \ref{type 2} and \ref{isolated type}, that if $g$ involves only $x_i$ and $x_j$'s that are adjacent to $x_i$, then either $C_G (g) < C_G (x_i)$ or $C_G (g)$ is an isolated vertex.  This implies that $C_G (x_i)$ is an independent vertex. 

Suppose that $C_G (x_i)$ is an independent vertex of $\Gamma_{\mathcal{Z}} (G)$.  Let $j$ be a neighbor of $i$ in $\widebar{S}$.  If $C_G (x_j)$ is a subordinate vertex to $C_G (x_i)$ in $\Gamma_{\mathcal{Z}} (G)$, then $C_G (x_j) \le C_G (x_i)$.  Notice that if $k$ is adjacent to $j$ in $\widebar{S}$, then $x_j \in C_G (x_j)$ and so, $x_k \in C_G (x_i)$ and this implies that $k$ is adjacent to $i$ in $\widebar{S}$.  It follows that $j$ is a subordinate vertex to $i$ in $\widebar{S}$.  On the other hand, if $C_G (x_j)$ is an independent vertex in $\Gamma_{\mathcal{Z}} (G)$, then there exists $x_h \in C_G (x_i) \setminus C_G (x_j)$ and $x_k \in C_G (x_j) \setminus C_G (x_i)$.  It follows that $h$ is adjacent to $i$ and not $j$ and $k$ is adjacent to $j$ and not $i$.  We conclude that $i$ is an independent vertex in $\widebar{S}$.
%Working here.

Suppose $g \in G$ satisfies the hypotheses of Lemma \ref{type 2} and $E$ is empty.  Thus, for every $x_i$ for $i \in D$ there is some $x_j$ with $j \in D$ so that $x_j$ does not commute with $x_i$.  We see from Lemma \ref{type 2} that both $x_i$ and $x_j$ are contained in $C_G (g)$.  It follows that $C_G (g)$ will not be contained in any $C_G (x_i)$.  From Lemmas \ref{clique}, \ref{type 2}, and \ref{isolated type}, we see that it is not contained in $C_G (h)$ for any $h$ that is a product of the $x_i$'s that do not lie in $\langle g \rangle$. 
\end{proof}

We next consider the possibility that $\widebar{S}$ has a connected component that is a complete graph.

\begin{lemma}
Let $G = G(p,n,S)$.  Assume that no $i$ is an isolated vertex in $S$ (viewed as a graph).   If a connected component $X$ of $\widebar{S}$ is a complete graph, then the corresponding connected component in $\Gamma_{\mathcal{Z}} (G)$ is an isolated vertex.
\end{lemma}

\begin{proof}
Fix an element $i \in X$.  Note that $X$ is a clique in $\widebar{S}$.  Thus, if $g$ is an element of $G$ that only involves $x_j$ for $j \in X$, then Lemma \ref{clique} implies that $C_G (g) = C_G (x_i)$.  If $g$ involves any $x_j$ for $j$ outside of $X$, then the above lemmas show that $C_G (g)$ will not be contained in $C_G (x_i)$.  Also, we see that $C_G (x_j) = C_G(x_i)$ for $j \in X$.  It follows that $C_G (x_i)$ is an isolated vertex in $\Gamma_{\mathcal{Z}} (G)$.
\end{proof}

We consider the connected components of $\widebar{S}$ that are not complete graphs.

\begin{theorem}
Let $G = G(p,n,S)$.  Let $X$ be a connected component of $\widebar{S}$ that is not a complete graph, then $\Gamma_{\mathcal {Z}} (G)$ contains a connected component $X^*$ that consists of the $C_G(x_i)$'s for $i \in X$ and $C_G (g)$ for $g$ as in Lemmas \ref{clique} or \ref{type 2} where all of the $x_i$'s involved in $g$ lie in $X$ and in Lemma \ref{type 2} the set $E$ is not empty.  Furthermore, if $l$ is the diameter of $X$ as a connected component of the complement of $S$, then the diameter of $X^*$ is between $l$ and $l+2$.
\end{theorem}

\begin{proof}
By Corollary \ref{single 2}, the centralizers $C_G (x_i)$ for $i \in X$ will all lie in the same connected component of $\Gamma_{\mathcal{Z}} (G)$, and we call this connected component $X^*$.  Observe that Lemmas \ref{clique} and \ref{type 2} show that every vertex in this connected component is adjacent to a vertex of the form $C_G (x_i)$.  Hence, for any minimal path in $\Gamma_{\mathcal{Z}} (G)$, we can replace all of the non-endpoints with vertices of the form $C_G (x_i)$.   It is not difficult to see that the subgraph induced by the $C_G (x_i)$'s will be isomorphic to $\widebar{S}$.  The result now follows.  
\end{proof}

\subsection{Specific examples}

We now consider some specific examples.  We start with two examples to illustrate Theorem \ref{Z(G) eq p^4}.

\medspace
{\bf Example 1:}  Take $p$ to be any prime, $n = 4$ and take 
$$S_1 = \{ (1,3), (1,4), (2,4), (3,4) \}.$$  
Take $G_1 = G(p,4,S_1)$.  Observe that $|G_1| = p^8$ and $|G_1:Z(G_1)| = p^4$. In this case, $C_{G_1} (x_2)$ is the unique independent vertex in the nonisolated vertices and it and its neighbors make up the connected component of nonisolated vertices in $\Gamma_{\mathcal{Z}} (G_1)$.  It will have $p^3$ isolated vertices including $C_{G_1} (x_4)$.

\medspace
{\bf Example 2:}  Take $p$ to be any prime, $n = 4$ and take 
$$S_2 = \{ (1,4), (2,4), (3,4) \}.$$  
Take $G_2 = G(p,4,S_2)$.  Observe that $|G_2| = p^7$ and $|G_2:Z(G_2)| = p^4$.  

%It is an open project to classify $\Gamma_{\mathcal Z} (G)$ for $G = G(n,p,S)$ and a specific set $S$.

\medskip
We are ready to build the prescribed $p$-group $G$ where $\Gamma_{\mathcal Z} (G)$ has more than one nontrivial connected component.  For this we need to specify $n$ and specify the set $S$ in our construction $G(p,n,S)$.

\medskip
{\bf Example 3:}   Let $k \geq 1$ an integer, and $n_1,\dots, n_k$ integers with each $n_i > 1$.  Set $n^* = n_1 + \cdots + n_k + k$.  For each $i$, $1 \leq i \leq k$, let $$E_i = \{ \{j,j+1\} \,\, | \,\, n_1 + \cdots + n_{i-1} + i \leq j < n_1 + \cdots + n_i + i \}.$$  Let  
$$S^* = \{ \{i,j\} \,\, | \,\, 1 \leq i < j \leq n \} - \bigcup_{1 \leq i \leq k} E_i.$$  
Note that we are taking $S^*$ to be the complement of $\cup_{i=1}^k E_i$.  Let $G = G(p,n^*,S^*)$ where $p$ is a prime and $n^*$ and $S^*$ are specified as above.  We will use the notation $G = {\mathfrak G}_p (n_1,\dots,n_k)$ for these groups. 

To illustrate this example we give an example of specific numbers.  Fix $k=3$ and $n_1 = 5$, $n_2 = 7$ and $n_3 = 4$.  So $n = 5 + 7 + 4 + 3 = 19$.  We have
\begin{gather*}
	E_1 = \{ \{1,2\}, \{2,3\}, \{3,4\}, \{4,5\}, \{5,6\} \},\\
	E_2 = \{ \{7,8\}, \{8,9\}, \{9,10\}, \{10,11\}, \{11,12\}, \{12,13\}, \{13,14\} \},\\
	E_3 = \{ \{15,16\}, \{16,17\}, \{17,18\}, \{18,19\} \}.
\end{gather*}

And so $S$ consists of all $\{i,j\}$'s other than these exceptions.  This example would be denoted ${\mathfrak G}_p (5,7,4)$.  Note that ${\mathfrak G}_p (n_1,\dots,n_k)$ is nondegenerate for all cases except for ${\mathfrak G}_p (2)$.

Returning to the general case of $G = {\mathfrak G}_p (n_1,\dots,n_k)$, we see that $|G:Z(G)| = p^n$.  Viewing $\widebar{S}$ as a graph, the connected components are going to be the $E_i$'s and they are paths each of length at least $2$; so none of them are complete graphs.  Hence, each corresponds to a connected component of $\Gamma_{\mathcal{Z}} (G)$ that does not consist of isolated vertices.  We claim that it is not too difficult to see that the diameter will be the same.  When $k = 1$, we obtain a group whose commuting graph has a connected component has diameter $n_1$, and if $n_1 \ge 3$, then the commuting graph has isolated vertices.  It follows that the examples we have are different than the ones in \cite{camInd} and \cite{nobound}.  

We close that by noting that all of the groups constructed in this manner that have more than one connected component that is not complete have isolated vertices in $\Gamma_{\mathcal{Z}} (G)$.   It would be interesting to see if there is a way to construct groups with centralizer graphs that have multiple connected components where all of the connected components have diameter at least $2$.

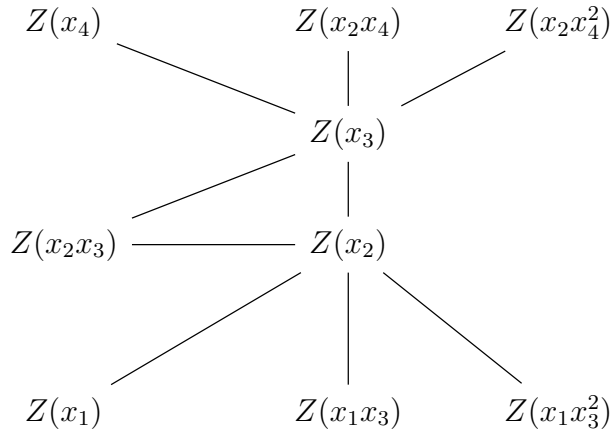
\begin{figure}[H]
	% https://q.uiver.app/#q=WzAsOSxbMiwyLCJcXGxhbmdsZSB4XzIgXFxyYW5nbGUgWihHKSJdLFsyLDEsIlxcbGFuZ2xlIHhfMyBcXHJhbmdsZSBaKEcpIl0sWzAsNCwiXFxsYW5nbGUgeF8xLHhfMiBcXHJhbmdsZSBaKEcpIl0sWzAsMiwiXFxsYW5nbGUgeF8yLHhfMyBcXHJhbmdsZSBaKEcpIl0sWzAsMCwiXFxsYW5nbGUgeF8zLHhfNCBcXHJhbmdsZSBaKEcpIl0sWzIsNCwiXFxsYW5nbGUgeF8xeF8zLCB4XzIgXFxyYW5nbGUgWihHKSJdLFszLDQsIlxcbGFuZ2xlIHhfMXhfM14yLCB4XzIgXFxyYW5nbGUgWihHKSJdLFsyLDAsIlxcbGFuZ2xlIHhfMnhfNCwgeF8zIFxccmFuZ2xlIFooRykiXSxbMywwLCJcXGxhbmdsZSB4XzJ4XzReMiwgeF8zIFxccmFuZ2xlIFooRykiXSxbMCwxLCIiLDAseyJzdHlsZSI6eyJoZWFkIjp7Im5hbWUiOiJub25lIn19fV0sWzAsMiwiIiwyLHsic3R5bGUiOnsiaGVhZCI6eyJuYW1lIjoibm9uZSJ9fX1dLFswLDMsIiIsMix7InN0eWxlIjp7ImhlYWQiOnsibmFtZSI6Im5vbmUifX19XSxbMSwzLCIiLDAseyJzdHlsZSI6eyJoZWFkIjp7Im5hbWUiOiJub25lIn19fV0sWzEsNCwiIiwwLHsic3R5bGUiOnsiaGVhZCI6eyJuYW1lIjoibm9uZSJ9fX1dLFswLDYsIiIsMCx7InN0eWxlIjp7ImhlYWQiOnsibmFtZSI6Im5vbmUifX19XSxbMCw1LCIiLDAseyJzdHlsZSI6eyJoZWFkIjp7Im5hbWUiOiJub25lIn19fV0sWzEsNywiIiwwLHsic3R5bGUiOnsiaGVhZCI6eyJuYW1lIjoibm9uZSJ9fX1dLFsxLDgsIiIsMCx7InN0eWxlIjp7ImhlYWQiOnsibmFtZSI6Im5vbmUifX19XV0=
	\[\begin{tikzcd}
		{Z(x_4)} && Z(x_2x_4) & Z(x_2x_4^2)  \\
		&& Z(x_3)  \\
		{Z(x_2x_3) } && {Z(x_2) } \\
		\\
		{Z(x_1) } &&Z(x_1x_3)  &Z(x_1x_3^2) 
%		{Z(x_4) = \langle x_3,x_4 \rangle Z(G)} && Z(x_2x_4) = {\langle x_2x_4, x_3 \rangle Z(G)} & Z(x_2x_4^2) = {\langle x_2x_4^2, x_3 \rangle Z(G)} \\
%		&& Z(x_3) = {\langle x_3 \rangle Z(G)} \\
%		{Z(x_2x_3) = \langle x_2,x_3 \rangle Z(G)} && {Z(x_2) = \langle x_2 \rangle Z(G)} \\
%		\\
%		{Z(x_1) = \langle x_1,x_2 \rangle Z(G)} &&Z(x_1x_3) =  {\langle x_1x_3, x_2 \rangle Z(G)} &Z(x_1x_3^2) = {\langle x_1x_3^2, x_2 \rangle Z(G)}
		\arrow[no head, from=2-3, to=1-1]
		\arrow[no head, from=2-3, to=1-3]
		\arrow[no head, from=2-3, to=1-4]
		\arrow[no head, from=2-3, to=3-1]
		\arrow[no head, from=3-3, to=2-3]
		\arrow[no head, from=3-3, to=3-1]
		\arrow[no head, from=3-3, to=5-1]
		\arrow[no head, from=3-3, to=5-3]
		\arrow[no head, from=3-3, to=5-4]
	\end{tikzcd}\]
	\caption{Component $C_1$ of the centralizer graph $\Gamma_{\mathcal Z} (G)$ for $G={\mathfrak G}_3(3,n_2,\dots,n_k)$.}
	\label{fig: p_group_example}
\end{figure}

\end{document}